\documentclass[preprint,aos]{imsart}
\RequirePackage{natbib}
\setattribute{journal}{name}{}

\RequirePackage[OT1]{fontenc}
\RequirePackage{amsthm,amsmath,natbib,graphicx,enumitem}
\RequirePackage[colorlinks,citecolor=blue,urlcolor=blue]{hyperref}
\RequirePackage{hypernat}
\usepackage{amssymb,tabularx,multicol,multirow,booktabs}
\usepackage[top=1in, bottom=1in, left=1in, right=1in]{geometry}
\usepackage{mhequ}
\usepackage[usenames,dvipsnames]{color}
\usepackage{epstopdf}
\RequirePackage{xr}
\setattribute{journal}{name}{}

\startlocaldefs
\numberwithin{equation}{section}
\theoremstyle{plain}
\endlocaldefs

\def \bbeta{\boldsymbol{\beta}}

\def \be{\begin{equs}}
	\def \ee{\end{equs}}

\def \I{\mathcal{I}}

\def \E{\mathbb{E}}
\def \P{\mathbb{P}}
\def \Q{\mathbb{Q}}

\def \fhat{\hat{f}}

\def \sumn{\sum_{i=1}^n}

\def \Pbeta {\mathbb{P}_{\boldsymbol{\beta},\lambda}}
\def \Pzero {\mathbb{P}_{\boldsymbol{\beta}=\mathbf{0},\lambda}}
\def \Ebeta {\mathbb{E}_{\boldsymbol{\beta},\lambda}}
\def \Ezero {\mathbb{E}_{\boldsymbol{\beta}=\mathbf{0},\lambda}}

\def \Vzero {\mathrm{Var}_{\boldsymbol{\beta}=\mathbf{0},\lambda}\left(HC(t)\right)}
\def \one {\boldsymbol{1}}
\def \delsup {\overline{\delta}}
\def \delinf {\underline{\delta}}
\def \dmax {\max_{i=1}^{n} d_i}
\def \Bin {\mathrm{Bin}}
\def \cdense {C_{\mathrm{dense}}}
\def \csparse {C_{\mathrm{sparse}}}
\def \cmax {C_{\mathrm{max}}}
\input binary_macro.tex

\begin{document}
%\doublespacing
\begin{frontmatter}
\title{Detection Thresholds for the $\beta$-Model on Sparse Graphs\thanksref{T1}}
\runtitle{Sharp Thresholds for $\beta$ model}
\thankstext{T1}{RM is a Stein Fellow at Department of Statistics, Stanford University. SS was partially supported by the William R. and Sara Hart Kimball Stanford Graduate Fellowship.}

\begin{aug}

\author{\fnms{Rajarshi} \snm{Mukherjee}\thanksref{m1}\ead[label=e2]{rmukherj@stanford.edu}},
\author{\fnms{Sumit} \snm{Mukherjee}\thanksref{m2}\ead[label=e1]{sm3949@columbia.edu}},
\and
\author{\fnms{Subhabrata} \snm{Sen}\thanksref{m1}\ead[label=e3]{ssen90@stanford.edu}}

\affiliation{Stanford University\thanksmark{m1}  and Columbia University \thanksmark{m2} }

\address{Department of Statistics\\
Sequoia Hall \\
390 Serra Mall, Stanford, CA- 94305. \\
\printead{e2}\\
\phantom{E-mail:\ }\printead*{e3}}

\address{Department of Statistics\\
1255 Amsterdam Avenue\\
New York, NY-10027. \\
\printead{e1}}

%\address{Address of the Third author\\
%usually few lines long\
%usually few lines long\\
%\printead{e3}\\
%\printead{u1}}
\end{aug}

\begin{abstract} In this paper we study sharp thresholds for detecting sparse signals in $\beta$-models for potentially sparse random graphs. The results demonstrate interesting interplay between graph sparsity, signal sparsity, and signal strength. In regimes of moderately dense signals, irrespective of graph sparsity, the detection thresholds mirror corresponding results in independent Gaussian sequence problems. For sparser signals, extreme graph sparsity implies that all tests are asymptotically powerless, irrespective of the signal strength. On the other hand, sharp detection thresholds are obtained, up to matching constants, on denser graphs. The phase transitions mentioned above are sharp. As a crucial ingredient, we study a version of the Higher Criticism Test which is provably sharp up to optimal constants in the regime of sparse signals. The theoretical results are further verified by numerical simulations.
\end{abstract}

\begin{keyword}[class=AMS]
\kwd[Primary ]{62G10}
\kwd{62G20}
\kwd{62C20}
%\kwd[; secondary ]{60K35}
\end{keyword}
\begin{keyword}
\kwd{Detection Boundary}
\kwd{Sparse Random Graphs}
\kwd{Beta Model}
\kwd{Higher Criticism}
\kwd{Sparse Signals}
\end{keyword}

\end{frontmatter}

\section{Introduction}
Real life networks are ubiquitous in the natural and social sciences. Over the last decade, a huge volume of data has been collected on the structure of large networks and dynamics on these systems. Specialists from a wide array of disciplines such as 
physics, biology, social science, computer science, and statistics have studied these complex systems using specialized techniques. This area has seen an explosion of research in the recent years, and the resulting 
confluence of diverse ideas has enriched this field of study.

From a statistical perspective, one avenue of research has focused on designing random graph models which exhibit some properties of real world networks such as a heavy-tailed degree distribution \citep{barabasi1999emergence}, a ``small world" characteristic  \citep{watts1998collective}, etc. There has also been some preliminary work on testing for goodness of fit of these models to real data \citep{bickel2011method}. 

A parallel line of research has focused on the design and analysis of algorithms for finding structural patterns or `motifs' in network data.  Finding these structures is of interest to practitioners in diverse research areas as they often represent functional relationships between the vertices.  
For example, in a social network setup, a group of vertices with a high edge density may represent a community of individuals who share some common characteristics, e.g. a common profession. 
Guarantees for relevant statistical algorithms are usually derived by analyzing their typical performance on specific random graphs. Prominent examples include finding a single planted community or finding two equal sized communities in a graph. The simpler problem of detecting the presence of these structural patterns is itself statistically challenging, as they are often rare, and thus one needs to design procedures which work in a low signal to noise ratio setup. This has been accomplished for the single community case in \citep{arias2013community, verzelen2015community}. This body of work connects to the broader statistical program of global testing against structured alternatives (\cite{Ingster4}, \cite{Jin1}, \cite{arias2005near},  \cite{arias2008searching}, \cite{addario2010combinatorial}, \cite{Jin2}, \cite{Ingster5}, \cite{Candes}, \cite{cai2014rate}, \cite{arias2015sparse}, \cite{mukherjee2015hypothesis}). %\textcolor{red}{(to arrange)!}

It is well-known in social science that vertices of a network are often differentially ``attractive", with the more popular vertices having a higher tendency to form edges. Finding the vertices with higher ``attractiveness" is of interest in a number of different contexts--- they often represent highly influential nodes and practitioners might wish to screen these vertices for subsequent study. In an online social network, these vertices might actually represent spam accounts which should be detected and removed. 

In this paper, we formulate and address the question of detecting differential attractiveness of vertices in networks. While this question is statistically simpler, we expect that mathematically probing the limits of detection can also provide non-trivial information about the corresponding estimation question. The detection problem is of independent interest in a different context--- assume that given a network dataset, one wishes to fit a stochastic block model to capture the community structure present in the data. It is well known that differential attractiveness of vertices can confound and affect model fits in this case \citep{karrer2011stochastic, yan2014model}. Detecting the presence of this structural pattern is therefore imperative in this scenario. We believe that similar testing questions could be useful even in such setups.

\subsection{Framework}
\label{section:framework}
Formally, we observe a labeled graph $\mathcal{G}= (V,E)$. We define $|V| =n$ and we are mainly interested in the asymptotic regime where the graph size $n\to \infty$. The graph $\mathcal{G}$ can be equivalently described by the adjacency matrix ${\bf{Y}}=(Y_{ij} : 1\leq i,j \leq n)$, where $Y_{ij} = 1$ if $\{i,j\} \in E$ and $Y_{ij}=0$ otherwise. We note that trivially, $Y_{ij}= Y_{ji}$ while $Y_{ii}= 0$, as we consider graphs without loops. 

%Let $\mathcal{G}_n$ denote the space of all labelled graphs on $n$ vertices. Since graphs and their adjacency matrices are in 1-1 correspondence, without loss of generality take $$\mathcal{G}_n=\Big\{{\bf Y}_{n\times n}\in \{0,1\}^{n^2}:Y_{ij}=Y_{ji}, Y_{ii}=0\Big\}.$$

Next, we specify a statistical model for the graphs in this context. 
For any ${\bbeta}:=(\beta_1,\cdots,\beta_n)\in\mathbb{R}^n$ and $\lambda\in [1,n]$, we define a probability distribution on the space of labelled graphs by setting
\begin{align}
\P_{{\bbeta},\lambda}({\bf Y})=\prod_{1\le i<j\le n}p_{ij}^{Y_{ij}}(1-p_{ij})^{Y_{ij}},\quad p_{ij}:=\frac{\lambda}{n}\frac{e^{\beta_i+\beta_j}}{1+e^{\beta_i+\beta_j}}. \label{eq:model_main}
\end{align}
Thus under the distribution $\P_{\bbeta,\lambda}$, the edges are mutually independent with $Y_{ij}\sim \Bin(1,p_{ij})$. The dependence of $\P_{\bbeta,\lambda}$ on $n$ will be suppressed throughout. $\beta_i$ should be interpreted as the ``attractiveness" of vertex $i$, while $\lambda$ controls the sparsity of the graph obtained. We will assume throughout $\lambda$ is known  and suppress the dependence of $\lambda$ on $n$. We note that the model \eqref{eq:model_main} is intimately related to the $p_1$-model of \cite{holland1981exponential} and the $\beta$-model of \cite{chatterjee2011random}. See Section \ref{section:betamodel} for background on these models and connections to \eqref{eq:model_main}. We discuss more on the choice of the model \eqref{eq:model_main} and the knowledge of $\lambda$ in Section \ref{section:discussion}.

In the context of the above model, we will formulate our problem as a goodness-of-fit type global null hypothesis testing problem against a structured hypothesis. To this end, we define the parameter space 
\begin{align}
\Xi(s,A):=\{\bbeta\in \mathbb{R}_{+}^n: |S(\bbeta)|=s, \beta_i\ge A, i\in S(\bbeta)\}, \label{eq:parameterspace}
\end{align}
where $S(\bbeta):=\{1\le i\le n:\beta_i\ne 0\}$ and $\mathbb{R}_{+} = [0, \infty)$. The vertices $i \in S(\bbeta)$ should be interpreted as the popular vertices. Since we expect such vertices to be rare, mathematically we consider the following sequence of hypothesis testing problems
\be 
H_0: \bbeta =\mathbf{0} \quad \textrm{vs.} \quad H_1: \bbeta \in \Xi(s_n, A_n) \subset \mathbb{R}_+^n\setminus \{\mathbf{0}\} \label{eqn:hypo}
\ee
for any pair of sequences $s_n, A_n$. Throughout we parametrize signal sparsity $s_n=n^{1-\alpha}$ with $\alpha \in (0,1)$. A statistical test for $H_0$ versus $H_1$  is a measurable $\{0,1\}$ valued function of the data $\bY$, with $1$ denoting the rejection of the null hypothesis $H_0$ and $0$ denoting the failure to reject $H_0$. The worst case risk of a test $T_n(\bY)$ is defined as 
\be \label{eqn:risk}
\mathrm{Risk}_n(T_n,\Xi (s_n, A_n))&:=\P_{\mathbf{0},\lambda}\left(T_n=1\right)+\sup_{\bbeta \in \Xi(s_n, A_n)}\P_{\bbeta, \lambda}\left(T_n=0\right). \label{eq:risk}
\ee
A sequence of tests $T_n$ corresponding to a sequence of model-problem pairs \eqref{eq:model_main}-\eqref{eqn:hypo}, is said to be asymptotically powerful (respectively asymptotically  powerless) against $\Xi(s_n ,A_n)$ if $$\limsup\limits_{n\rightarrow \infty}\mathrm{Risk}_n(T_n,\Xi(s_n, A_n)= 0\text{ (respectively }\liminf\limits_{n\rightarrow \infty}\mathrm{Risk}_n(T_n,\Xi(s_n, A_n)=1).$$

\subsection{Background on the $\beta$-model for random graphs}
\label{section:betamodel}
The study of degree sequences of network data has a long and rich history (\cite{holland1981exponential}, \cite{fienberg1981categorical}, \cite{robins2007introduction}, \cite{goodreau2007advances}, \cite{barvinok2013number}). The simplest model for networks based on the degree sequence is an exponential family with the degree sequence as its sufficient statistic. This model is a special case of the $p_1$ model of \cite{holland1981exponential}, and will be called the $\beta$-model, following recent terminology of \cite{chatterjee2011random}. We will not survey the vast literature on the $p_1$ model and its applications and instead refer the reader to \cite{blitzstein2011sequential} for detailed references. The undirected model can be equivalently described as follows: given a vector $\bbeta = (\beta_1, \cdots, \beta_n) \in \mathbb{R}^n$, we form a graph on the vertex set $V$ with $|V|= n$ and edges are added independently with probability 
\begin{align}
p_{ij} =\frac{ \exp{(\beta_i + \beta_j)}}{1 +  \exp{(\beta_i + \beta_j)}}. \label{eq:model_original}
\end{align}
We note that the model \ref{eq:model_main} and \ref{eq:model_original} are identical except for the leading $\lambda/n$ factor. This factor is introduced to control the sparsity of the graph. The original $\beta$-model leads to dense graphs, where a graph on $n$ vertices has $O_P(n^2)$ edges. However, it is widely accepted that real networks are seldom dense, and thus to address this issue, we introduce an additional parameter $\lambda$. 

The original model \eqref{eq:model_original} has been studied widely in recent years. It is known to be the maximum entropy distribution given the degree distribution (see \cite{blitzstein2011sequential}). \cite{lauritzen2002rasch, lauritzen2008exchangeable} characterized these models as natural models for \textit{weakly summarized} exchangeable binary arrays, i.e. arrays with distributions determined by row and column totals. Statistical analysis of the model \eqref{eq:model_original} has also received a lot of attention --- the existence and consistency of MLE's in these models was examined by \cite{chatterjee2011random} and \cite{rinaldo2013maximum}, and normal fluctuation for the MLE was established in \cite{yan2013central}. \cite{perry2012null} also study a general model which includes the $\beta$-model as a special case. Finite sample analysis of these models using Markov Bases has also been explored in \cite{petrovic2010algebraic, hara2010connecting, ogawa2013graver, yan2013central}. {
In a different direction, \cite{karwa2016differential} studies differentially private parameter estimation in the $\beta$ model, while  \cite{hillar2013entropy} and \cite{yan2015normality} generalize the $\beta$ model to weighted graphs. We also refer the reader to \cite{yan2016asymptotics} for asymptotic results on a general family of network models, which includes the $\beta$ model as a special case.}

\subsection*{Notation} For any $n \in \mathbb{N}$, we let $[n]=\{1,\ldots,n\}$. For any $i\in [n]$ we denote the degree of vertex $i$ by $d_i:=\sum_{j=1}^nY_{ij}=\sum_{j\ne i}Y_{ij}.$  Throughout $\Bin(n,p)$ will stand for a generic binomial random variable with $n\in \mathbb{N}$ trials and success probability $p \in [0,1]$. The results in this paper are mostly asymptotic in nature and thus requires some standard asymptotic  notations. If $a_n$ and $b_n$ are two sequences of real numbers then $a_n \gg b_n$ (and $a_n \ll b_n$) implies that ${a_n}/{b_n} \rightarrow \infty$ (respectively ${a_n}/{b_n} \rightarrow 0$) as $n \rightarrow \infty$. Similarly $a_n \gtrsim b_n$ (and $a_n \lesssim b_n$) implies that $\liminf{{a_n}/{b_n}} = C$ for some $C \in (0,\infty]$ (and $\limsup{{a_n}/{b_n}} =C$ for some $C \in [0,\infty)$). Alternatively, $a_n=o(b_n)$ will also imply $a_n \ll b_n$ and $a_n=O(b_n)$ will imply that $\limsup{{a_n}/{b_n}} =C$ for some $C \in [0,\infty)$). We write $a_n =\Theta(b_n)$ if both $a_n=O(b_n)$ and $b_n=O(a_n)$. We write $a_n \sim b_n$ if $\lim \frac{a_n}{b_n}\rightarrow 1$. For any fixed  tuple $v$ of real numbers, $C(v)$ will denote a constant depending on elements of $v$ only. Also, throughout we drop the subscript $n$ whenever it is understood that $s, A, \lambda$ are allowed to vary with $n$.

\section{Tests}
\label{section:tests}

The tests used in this paper for the purpose of asymptotically sharp detection are all based on the degree vector $(d_1,\ldots,d_n)$. All the tests aim to capture the idea that the coordinates of the degree vector are stochastically increasing according to the corresponding coordinates of $\bbeta$.
\begin{description}[align=left]\itemsep15pt
	\item [\textbf{Total Degree Test} :] This test is based on the total degree in the observed graph i.e. $\sum_{i=1}^n d_i$. The test rejects when the observed total degree is large. The calibration of this test can be achieved by looking at the behavior of $\sum_{i=1}^n d_i$ under the null hypothesis in \eqref{eqn:hypo}. More precisely, by the Total Degree Test we mean a testing procedure which rejects when $\sum_{i=1}^n d_i-\frac{\lambda (n-1)}{2}$ is large (See proof of Theorem \ref{thm:dense}). 
	
	\item [\textbf{Maximum Degree Test} :] This test is based on the maximum degree in the observed graph i.e. $\max\limits_{i=1}^n d_i$.  The distribution of  $\max\limits_{i=1}^n d_i$ under the null hypothesis for $\lambda\gg \log^3{n}$ is standard \citep{bollobas}. However, for the theoretical calibration of an asymptotically powerful test, we simply need suitable control over the tail of the maximum degree (See proof of Theorem \ref{thm:max_degree_test}). In the rest of the paper, by Maximum Degree Test we shall mean a testing procedure that rejects when the observed maximum degree is large.

	\item [\textbf{Higher Criticism Test} :] This test is based on suitably scanning over centered and scaled survival statistics of the degree vector. More precisely, for any $t>0$ we let 
	\be 
	HC(t):=\sum_{i=1}^n \left(\I\left(D_i>t\right)-\Pzero\left(D_i>t\right)\right),
	\ee
	$$D_i=\frac{d_i-\frac{\lambda}{2n}(n-1)}{\sqrt{(n-1)\frac{\lambda}{2n}\left(1-\frac{\lambda}{2n}\right)}}, \quad i=1,\ldots,n.$$
	We then construct a version of the higher criticism test as follows. Define
	\be 
	HC:=\sup\left\{GHC(t):=\frac{HC(t)}{\sqrt{Var_{\boldsymbol{\beta}=\mathbf{0},\lambda}\left(HC(t)\right)}}, t\in \{\sqrt{2r\log{n}}:r\in (0,5)\}\cap \mathbb{N}\right\}. 
	\ee
	By Higher Criticism Test we then mean a testing procedure that rejects when the observed value of $HC$ defined above is large (See proof of Theorem \ref{thm:sparse} \ref{thm:sparse_hopeful} \ref{thm:sparse_upperbound}). Indeed this test is designed along the tradition of tests introduced and studied in recent history of sparse signal detection in sequence and regression models (\cite{Jin1,Jin2}; \cite{Candes}; \cite{arias2015sparse}; \cite{mukherjee2015hypothesis}; \cite{barnett2016ghc}). However, in spite of the similarity in philosophy of construction of the test, sharp analysis of the test turns out to be subtle due to the presence of dependence among the degree sequence. To our best knowledge, this is the first instance of sharp analysis of the Higher Criticism Test for a dependent Binomial sequence. This is one of the main technical contributions of the paper, since tackling these dependence issues is fundamentally different from analyzing a weakly dependent Gaussian sequence \citep{Candes}. A glance at our proof reveals that the case of sparser graphs, with $\lambda$ behaving polylogarithmically in $n$ is relatively easier. This is intuitively reasonable, since sparse graphs implies ``weaker" dependence among the degrees. Since our results are valid for any $\lambda\gg \log{n}$, more care is needed to attend to the dependence structure in the problem. 
\end{description}

\section{Main Results}
\label{section:main_results}
We divide the main results according to signal sparsity $\alpha \in (0,1)$. Our first theorem corresponds to the dense regime of signal sparsity i.e. $\alpha\leq \frac{1}{2}$.
\begin{theorem}\label{thm:dense}
Suppose $\alpha \le \frac{1}{2}$ and let
\be
\cdense(\alpha) = \frac{1}{2}-\alpha.
\ee

\begin{enumerate}[label=\textbf{\roman*}.]
\item\label{thm:dense_upper}
The Total Degree Test is asymptotically powerful if \be\label{eq:upper_dense}
%\lim_{n\rightarrow\infty}s\tanh(A)\sqrt{\frac{\lambda}{n}}=\infty,
\tanh(A)\geq \frac{n^{-r}}{\sqrt{\lambda}}, \quad r<\cdense(\alpha).
\ee
% then there exists a sequence of tests $\phi_n$ such that
%$$\E_{\beta=0}\phi_n+\sup_{\beta\in \Xi(s_n,A_n)}\E_{\beta,\lambda}\phi_n=0.$$

\item\label{thm:dense_lower}
All tests are asymptotically powerless if 
\be\label{eq:lower_dense}
%\lim_{n\rightarrow\infty}s\tanh(A)\sqrt{\frac{\lambda}{n}}=0,
\tanh(A)\leq \frac{n^{-r}}{\sqrt{\lambda}}, \quad r>\cdense(\alpha).
\ee
%then for any sequence of tests $\phi_n$ we have
%$$\E_{\beta=0}\phi_n+\sup_{\beta\in \Xi(s_n,A_n)}\E_{\beta,\lambda}\phi_n\ge 1.$$

\end{enumerate}

\end{theorem}

We note that one can actually state a slightly stronger result which dictates that all tests are asymptotically powerless when $\lim_{n\rightarrow\infty}s\tanh(A)\sqrt{\frac{\lambda}{n}}=0$ while the Total Degree test is asymptotically powerful whenever there exists a diverging sequence $t_n \rightarrow \infty$ such that $s\tanh(A)\sqrt{\frac{\lambda}{n}}\gg t_n$.

Our next result characterizes the detection thresholds in the sparser regime i.e. $\alpha>\frac{1}{2}$. Unlike the denser regime $(\alpha \leq \frac{1}{2})$, the results here depend on the level of graph sparsity $\lambda$. In particular for small $\lambda$'s (very sparse graphs) all tests turn out to be asymptotically powerless irrespective of the signal strength.
\begin{theorem}\label{thm:sparse}
Suppose $\alpha\in (1/2,1)$, $\theta:=\lim_{n\rightarrow \infty}\frac{\lambda}{2n}$, and let
\be
\csparse(\alpha) = \begin{cases}
16(1-\theta)\left(\alpha-\frac{1}{2}\right), &\frac{1}{2}< \alpha<\frac{3}{4}\\
16(1-\theta)\left(1-\sqrt{1-\alpha}\right)^2, & \alpha\geq \frac{3}{4}.
\end{cases}
\ee

\begin{enumerate}[label=\textbf{\roman*}.]

\item \label{thm:sparse_hopeless} Assume $\lambda\ll \log{n}$. Then all tests are asymptotically powerless irrespective of $A$.
\item \label{thm:sparse_hopeful} Assume $\lambda \gg \log{n}$.
\begin{enumerate}[label=\textbf{\alph*})]
\item\label{thm:sparse_upperbound} The Higher Criticism Test is asymptotically powerful if
\begin{align*}
\tanh(A)\ge \sqrt{\frac{C^*\log n}{\lambda}}, \quad C^*>\csparse(\alpha).
\end{align*}  
% then there exists a sequence of tests $\phi_n$ such that
%$$\E_{\alpha=0}\phi_n+\sup_{\beta\in \Xi(s_n,A_n)}\E_{\beta,\lambda}\phi_n=0.$$

\item\label{thm:sparse_lowerbound}
All tests are asymptotically powerless if
\begin{align*}
\tanh(A)\leq \sqrt{\frac{C^*\log n}{\lambda}}, \quad C^*<\csparse(\alpha).
\end{align*}  
% then for any sequence of tests $\phi_n$ we have
%$$\E_{\beta=0}\phi_n+\sup_{\beta\in \Xi(s_n,A_n)}\E_{\beta,\lambda}\phi_n\ge 1.$$
\end{enumerate}

\end{enumerate}
\end{theorem}
%Some remarks are in order regarding the results presented in Theorem \ref{thm:sparse}. 

\begin{remark}
The detection limits presented above are reminiscent of those in \cite{mukherjee2015hypothesis}. However, where results in \cite{mukherjee2015hypothesis} in essence correspond to $\lambda=n$, Theorem \ref{thm:sparse} presents a richer class of phase transitions by taking into account a continuous family of parameters $\theta:=\lim \lambda/2n$. 
%In the language of a balanced regression setting which observes $Y_{ij}\sim \Bin(1,\frac{\lambda}{n}\frac{e^{\beta_i+\beta_j}}{1+e^{\beta_i+\beta_j}})$ independently over all pairs $(i,j)\in [1:n]^2$ such that $i \neq j$, a similar detection limit to that in Theorem \ref{thm:sparse} indeed holds. 
However, the proof of such a result is simpler than that presented here due to the lack of dependence structure in traditionally studied regression models. Of course, the essence of our proof rests on understanding and carefully isolating the dependence structure present, and thereby reducing the calculations to a simpler independent binomial sequence problem. The subtlety of the proof relies on keeping track of the errors accumulated in such reductions. Finally, although the proof of the lower bound (Theorem \ref{thm:sparse}\ref{thm:sparse_hopeful}\ref{thm:sparse_lowerbound}) follows a general scheme of bounding a truncated second moment, deciding on the truncation event is subtle. %Moreover, the proof of the upper bound (Theorem \ref{thm:sparse}\ref{thm:sparse_hopeful}\ref{thm:sparse_upperbound}) regarding the performance of the Higher Criticism Test is substantially more involved compared to an independent Binomial sequence problem. Finally, the proof idea in \cite{mukherjee2015hypothesis} relies on a Koml{\'o}s-Major-Tusn{\'a}dy type strong coupling of Binomial sequences with appropriate Gaussian random variables. In this case, we simplify the proof arguments
\end{remark}	

\begin{remark}
{The result that all tests are asymptotically powerless when $\lambda \ll \log{n}$ is intuitive since
 according to our parametrization  $p_{ij}\in [\frac{\lambda}{2n}, \frac{\lambda}{n}]$ regardless of the membership of the vertices in $S$ or $S^c$.  In order to achieve non-trivial detection boundary for $\lambda\ll \log{n}$ one might need to work with a different parametrization of graph sparsity which allows for varying graph density between the null and alternative hypotheses. }
 %In particular, we expect that the parametrization  $\mathrm{logit}{(p_{ij}(\lambda))}=\log{\frac{\lambda}{n}}+\beta_i+\beta_j$ has richer detection threshold behavior. 
\end{remark}

Theorem \ref{thm:sparse} establishes the sharp optimality of the Higher Criticism Test in the sparse signal regime. Another commonly studied test of practical interest is the Maximum Degree Test introduced in Section \ref{section:tests}. Our next theorem characterizes the detection limits of the Maximum Degree Test.

\begin{theorem}\label{thm:max_degree_test}
Suppose $\alpha\in (1/2,1)$, $\theta:=\lim_{n\rightarrow \infty}\frac{\lambda}{2n}$, and let
\be
\cmax(\alpha) = 
16(1-\theta)\left(1-\sqrt{1-\alpha}\right)^2.
\ee

\begin{enumerate}[label=\textbf{\roman*}.]
	\item\label{thm:max_degree_upper}
	Assume $\lambda \gg \log{n}$. Then the Maximum Degree Test is asymptotically powerful if \be
	%\lim_{n\rightarrow\infty}s\tanh(A)\sqrt{\frac{\lambda}{n}}=\infty,
	\tanh(A)\ge \sqrt{\frac{C^*\log n}{\lambda}}, \quad C^*>\cmax(\alpha).
	\ee
	% then there exists a sequence of tests $\phi_n$ such that
	%$$\E_{\beta=0}\phi_n+\sup_{\beta\in \Xi(s_n,A_n)}\E_{\beta,\lambda}\phi_n=0.$$
	
	\item\label{thm:max_degree_lower}
		Assume $\lambda \gg \log^3{n}$. Then the Maximum Degree Test is asymptotically powerless if 
		\be
		%\lim_{n\rightarrow\infty}s\tanh(A)\sqrt{\frac{\lambda}{n}}=\infty,
		\tanh(A)\le \sqrt{\frac{C^*\log n}{\lambda}}, \quad C^*<\cmax(\alpha).
		\ee
	
\end{enumerate}

%Then the {Max Degree Test} is asymptotically powerful iff
%\begin{align*}
%\tanh(A)\ge \sqrt{\frac{C^*\log n}{\lambda}}, \quad C^*>C(\alpha).
%\end{align*}  
\end{theorem}
{The results in Theorem \ref{thm:max_degree_test} establish that for $\lambda \gg \log^3{n}$, the Maximum Degree Test is sharp optimal iff $\alpha>3/4$, and looses out to the Higher Criticism Test in the regime of $\alpha \in (1/2,3/4)$. Although this result is parallel to those observed in sparse detection problems for independent Gaussian and Binomial sequence models \citep{Candes,mukherjee2015hypothesis}, the proof of this fact is substantially more involved. This is in particular true about the proof of the lower bound part in Theorem \ref{thm:max_degree_test}\ref{thm:max_degree_lower}. Also note that Theorem \ref{thm:max_degree_test}\ref{thm:max_degree_upper} hold  as soon as $\lambda \gg \log{n}$. However, in order to prove the lower bound part in theorem above we crucially make use of the null distribution of $\max_{i=1}^n d_i$, which is readily available for $\lambda \gg \log^3{n}$ \citep{bollobas}. The situation becomes highly subtle for $\log{n}\ll \lambda \lesssim \log^3{n}$. In particular, we were able to argue that for $\log{n}\ll \lambda \lesssim \log^3{n}$, if one considers the Maximum Degree Test that rejects when $\dmax>np_n+\sqrt{\delta_n np_nq_n\log{n}}$, where $p_n=\lambda/2n$, $q_n=1-p_n$, and $\delta_n$ is some sequence of real numbers, then such tests are asymptotically powerless as soon as $C^*<\cmax(\alpha)$ defined above, if $\limsup \delta_n \neq 2$.  Lowering the requirement of $\lambda \gg \log^3{n}$ to $\lambda \gg \log{n}$ in this case requires a second moment argument along with Paley-Zygmund Inequality.  We refer to Appendix \ref{sec:appendix_maxdegree} for more details. On the other hand, the case when $\limsup \delta_n = 2$ is extremely challenging, and the result of the testing problem depends on the rate of convergence of $\delta_n$ to $2$ along subsequences. We leave this effort to future ventures.}
	
%We do believe though that it is possible to get the requirement down to $\lambda \gg \log{n}$ even in the necessity part without explicitly using the null distribution of the maximum degree. However, we do not pursue such avenues in this paper.}

\section{Simulation Results}
\label{section:simulation}

We now present the results of some numerical experiments in order to demonstrate the behavior of the various tests in finite samples. To put ourselves in context of our asymptotic analysis, we chose to work with $n=100$. Since our theoretical results depend both on the signal sparsity and graph sparsity, we divide our simulation results accordingly. In each of the situations, we compare the power of the tests (Total Degree, Maximum degree, and Higher Criticism respectively) by fixing the levels at $5\%$. In particular, we generate the test statistics $100$ times under the null and take the $95\%$-quantile as the cut-off value for rejection. The power against different alternatives are then obtained empirically from $100$ repeats each. %In addition each of the simulations are presented to represent power along different scales of sparsity and signal strength.

Our first set of simulations corresponds to the dense regime i.e. $\alpha\leq \frac{1}{2}$.  In figure \ref{fig:dense_signal}, we plot the power of each of the three tests for a range of signal sparsity-strength pairs $(\alpha,r)$, where $\alpha \in (0,\frac{1}{2})$ with increments of size $0.025$ and the signal strength is given by $A=\frac{n^{-r}}{\sqrt{\lambda}}$ with $r\in (0,\frac{1}{2})$ with increments of size $0.025$, and $\lambda=25$. In addition, we plot the theoretical detection boundary given by $r=\frac{1}{2}-\alpha$ in red. As dictated by our theoretical results, the phase transitions are clear in the simulations as well. In particular, we observe that the Total Degree Test performs better in the dense regime. The Higher Criticism Test seems to have some power, the Maximum Degree Test fails have any power in this regime of sparsity.

Our second set of simulations corresponds to the sparse regime i.e. $\alpha> \frac{1}{2}$.  In this case, following the theoretical predictions, we divide our simulations based on the graph sparsity $\lambda$ as well. In figure \ref{fig:sparse_signal}, we plot the power of each of the three tests for a range of signal sparsity-strength pairs $(\alpha,r)$ for three different values of $\lambda$, namely $\lambda=2,10,$ or $25$ respectively.  Specifically we choose for signal sparsity index varying between $\alpha \in (\frac{1}{2},1)$  with increments of size $0.025$, the signal strength to be $A=\sqrt{\frac{C\log{n}}{\lambda}}$ with $C\in (0,16)$ with increments of size $0.5$. In addition, we plot the theoretical detection boundary given by $r=\csparse(\alpha):=16(1-\theta)\{(\alpha-1/2)\I(\alpha<3/4)+(1-\sqrt{1-\alpha})^2\I(\alpha \geq 3/4)\}$ with $\theta=\lambda/2n$, in red. To distinguish between the performance of the Higher Criticism Test and the Maximum Degree Test, we also plot the theoretical detection boundary of the Maximum Degree Test in cyan.  The simulation results seem to match the theoretical predictions. In particular, when is $\lambda<\log{n}=3$, none of the tests seems to have any power in detecting any signal presented. In contrast for $\lambda$ much larger than $\log{n}$, the empirical performance of the Higher Criticism Test and Maximum degree Test also follows the theoretical guarantees. 

\begin{figure}
\begin{center}
\includegraphics[width=13 cm,keepaspectratio]{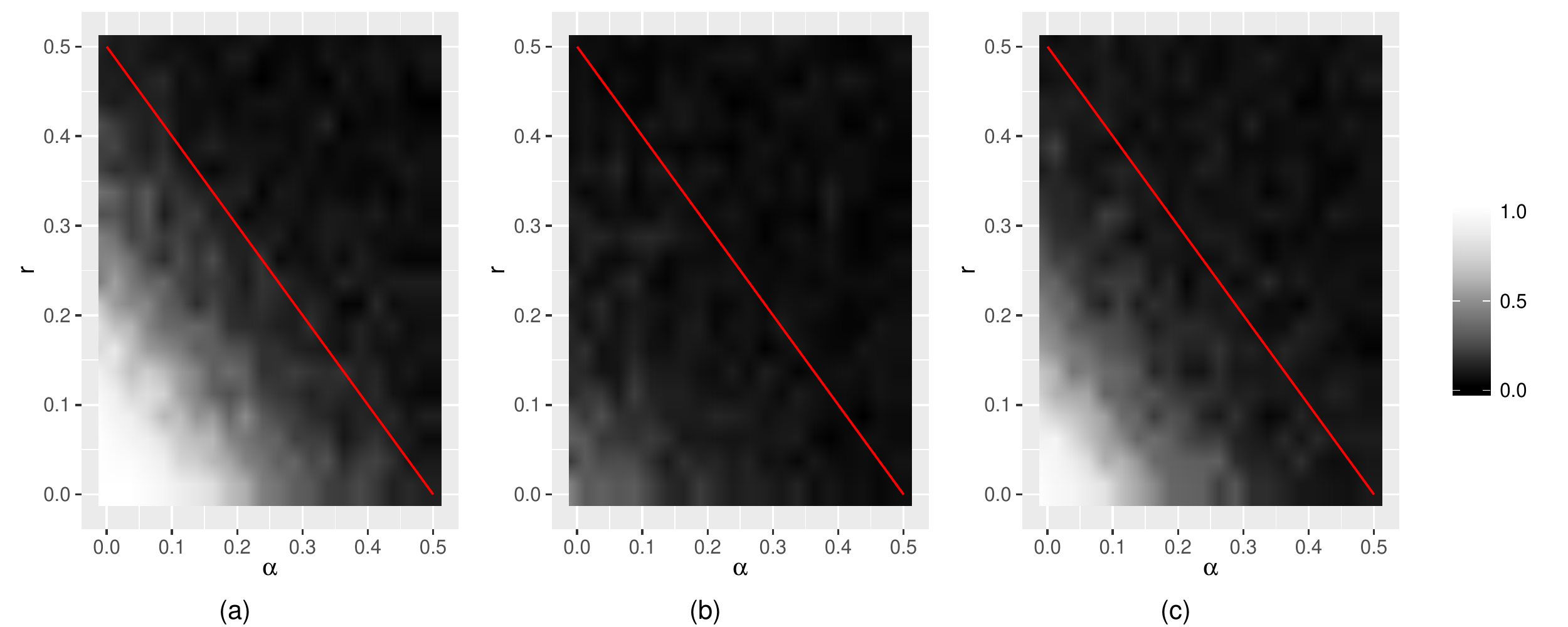}
\caption{The power of testing procedures in the dense signal setup. (a) shows the power of the Total degree test, (b) plots the max-degree while (c) plots the power of the GHC statistic \label{fig:dense_signal}. The theoretical detection threshold is drawn in red.}
\end{center}
\end{figure}

\begin{figure}
\begin{center}
\includegraphics[width=15 cm,keepaspectratio]{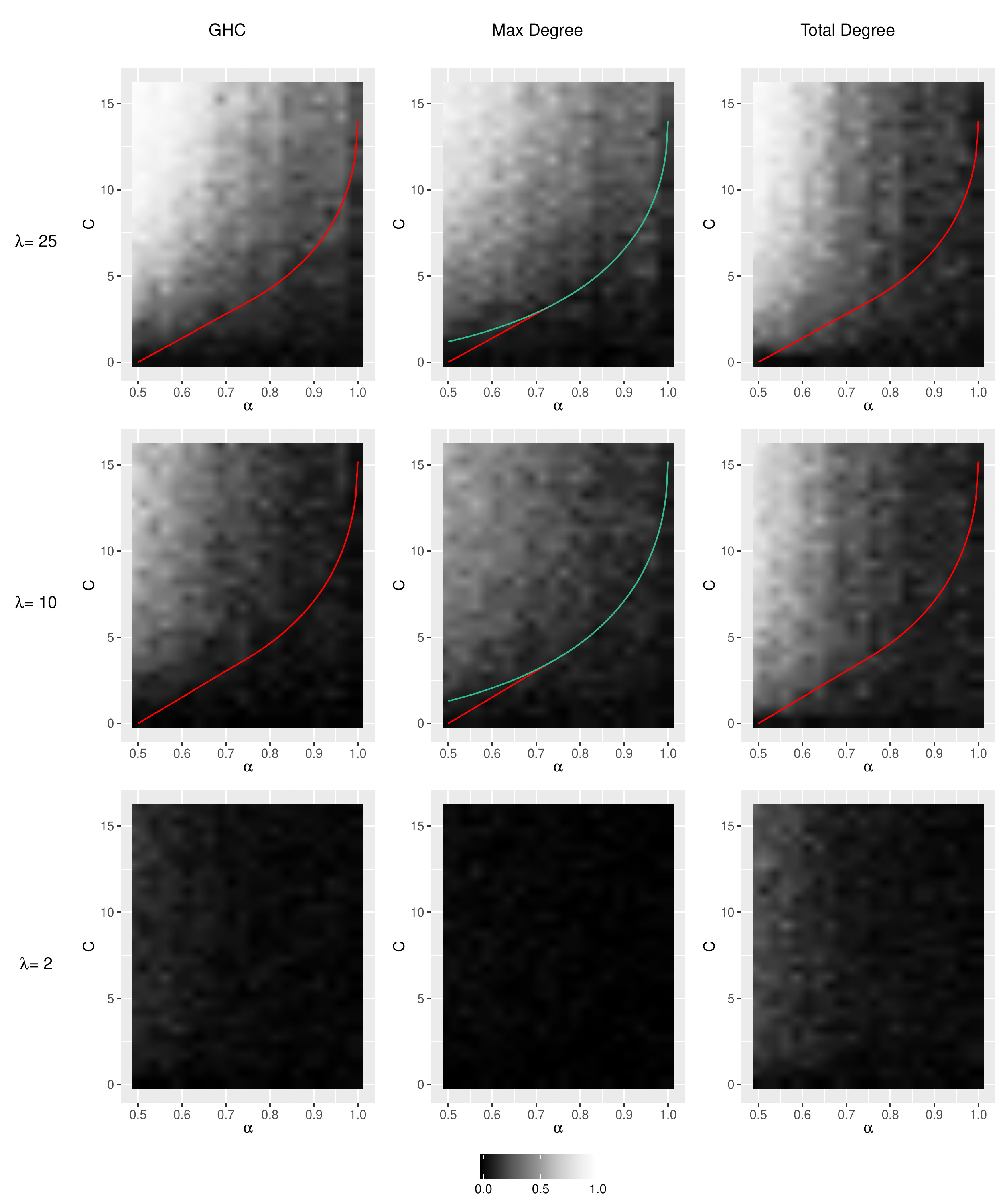}
\end{center}
\caption{The power of the testing procedures in the sparse signal setup. The theoretical detection thresholds are drawn in red, while the thresholds for the Maximum Degree Test are drawn in cyan.\label{fig:sparse_signal}}
\end{figure}

\section{Discussion}
\label{section:discussion}
In this paper we formulate and study the problem of detecting the presence of differentially attractive nodes in networks through a version of the $\beta$-model on sparse graphs. Our mathematical results are sharp in all regimes of signal strength as well as graph and signal sparsity. The model can be further generalized, where instead of using an $\mathrm{expit}$ function, the connection probabilities between node $i$ and $j$ is given by $\frac{\lambda}{n}\psi(\beta_i+\beta_j)$ for some function $\psi$ which is the distribution function of a symmetric random variable, i.e. $\psi(x)+\psi(-x)=1$, and satisfies some reasonable smoothness type regularity conditions. This in particular, will include the probit link  $\psi(x)=\int\limits_{-\infty}^xe^{-t^2/2}/\sqrt{2\pi}dt$. Thereafter, similar to \cite{mukherjee2015hypothesis}, we believe that one can obtain parallel results here as well, with the sharp constants in the problem changing according to the function $\psi$. 

{All of our results assume the knowledge of the graph sparsity parameter $\lambda$. For unknown $\lambda$, one needs to include supremum over $\lambda$  in the definition of risk in \eqref{eq:risk}. Following intuitions from \cite{arias2013community}, we believe that the nature of detection thresholds for unknown $\lambda$ remains the same for $\alpha>\frac{1}{2}$.  However it might change for the dense signal regime $\alpha\leq \frac{1}{2}$, where it might be more prudent to use a degree variance type test instead of the total degree test. For the sparse regime $(\alpha>\frac{1}{2})$, we intuitively believe that as a first step, one can choose a vertex at random and estimate $\lambda$ by the proportion of its neighbors. Since $s \ll \sqrt{n}$ in this regime of signal sparsity, a randomly chosen vertex belongs to the set of vertices with corresponding $\beta_i=0$ with high probability, and as a result it is not too hard to show that the resulting estimate is consistent in ratio scale. One can then delete this randomly chosen vertex from the network and consider the detection problem on the remaining $n-1$ vertices, and use the estimated $\lambda$ as a plug-in. It indeed remains a question to study how the performance of the corresponding testing procedures change under this plug-in principle. Although the analysis of the Maximum Degree Test should not be too difficult to carry out using techniques similar to those carried out here, the analysis of the Higher Criticism Test is extremely subtle and is beyond the scope of the current paper. }
%Finally, note that the decision theoretic framework described above is different from the case where the graph sparsity $\lambda$ is allowed to vary between the null and the alternative, and therefore needs to be incorporated by taking a supremum over possible ranges of $\lambda$ in \eqref{eq:risk}.

Although we only study the sparse detection problem in this paper, the rich phase transitions observed hint at similar complexities in further inferential questions. Of particular interest is the subset selection problem, which corresponds to the identification of vertices of differential attractiveness. We plan to address sharp analyses of such inferential questions in future papers.
%\tcr{unknown $\lambda$, can be generalized to connection probability being $\phi(\beta_i+\beta_j)$ with $\phi$ satisfying some regularity assumptions, further project: support recovery and sparse estimation.}

\section{Proofs of main results}
In this section, we prove the main results in this paper. 
\subsection{Some properties of Binomial distribution}

In this section, we collect some properties of the binomial distribution which will be used throughout this paper. The first lemma is a simple change of measure argument and the proof is omitted.

%In order analyze we set $h(x)=f(x)/g(x)$ and make use of the following lemma.
\begin{lemma}\label{lemma:binomial_changeof_measure}
If $X\sim \Bin(n,p)$ then for any $a>0$ and Borel subset of $B$ of $\mathbb{R}$
\be 
\E\left(a^X\one_{X\in B}\right)=(ap+(1-p))^n \P\left(\Bin\left(n,\frac{ap}{ap+(1-p)}\right)\in B\right).
\ee
\end{lemma}

The next lemma is crucial and the proof is deferred to Appendix \ref{sec:proof_of_binomial_lemma}.  

 \begin{lemma}\label{lemma:binomial_collection} 
%Let $X \sim \Bin(n,p)$ and $q=1-p$. Then we have the following estimates with $h>0$.
\begin{enumerate}

\item [(a)]Suppose $X_n\sim \Bin(n,p_n)$ such that $n\min(p_n,1-p_n)\gg \log n$, and  $\{C_n\}_{n\ge 1}$ is a sequence of reals converging to $C>0$. Then we have 

\item[(i)]
\be 
 \P\left(X_n= np_n+C_n\sqrt{np_n(1-p_n)\log n}\right)&=\frac{1}{\sqrt{np_n}}n^{-\frac{C^2}{2}+o(1)},
\ee
and
\item[(ii)]

$$\lim_{n\to \infty} \P\left(X_n\ge np_n+C_n\sqrt{np_n(1-p_n)\log n}\right)= n^{-\frac{C^2}{2}+o(1)}.$$

\item[(b)]
Suppose further that $Y_n\sim \Bin(b_n,p_n')$ is independent of $X_n$, and
\be
\limsup_{n\rightarrow\infty}p_n<1,\,\,\,\,\,\,
\limsup_{n\rightarrow\infty}\frac{b_n}{\sqrt{n}}<\infty,\,\,\,\,\,\,
\lim_{n\rightarrow\infty}\frac{p_n'}{p_n}=1.
\ee
Then we have
\item[(i)]
\begin{align*}
\P(X_n+Y_n=np_n+C_n\sqrt{n p_n(1-p_n)\log n})= \frac{1}{\sqrt{np_n}} n^{-\frac{C^2}{2}+o(1)},
\end{align*}
and 
\item[(ii)]
\be
 \P(X_n + Y_n \geq np_n+C_n\sqrt{n p_n(1-p_n)\log n}) = n^{-\frac{C^2}{2}+o(1)}. 
\ee

\end{enumerate}
\end{lemma}

\subsection{Proof of Theorem \ref{thm:dense}}
We prove here a slightly stronger result which dictates that all tests are asymptotically powerless when $\lim_{n\rightarrow\infty}s\tanh(A)\sqrt{\frac{\lambda}{n}}=0$ while the Total Degree test is asymptotically powerful whenever there exists a diverging sequence $t_n \rightarrow \infty$ such that $s\tanh(A)\sqrt{\frac{\lambda}{n}}\gg t_n$.
We begin by stating the following elementary lemma.
\begin{lemma}\label{lemma_totdeg_var}
For any $\bbeta\in \mathbb{R}_+^n,\lambda>0$ we have 
\begin{align*}
\label{eq:var}Var_{\bbeta,\lambda}(\sum_{i=1}^nd_i)\le &2n\lambda.
\end{align*}
\end{lemma}

\begin{proof}
The above bound follows on noting that
\\

$
Var_{\bbeta,\lambda}(d_i)\le \sum_{j\ne i} \E_{\bbeta,\lambda}Y_{ij}\le \lambda,\quad 
Cov_{\bbeta,\lambda}(d_i,d_j)=Var_{\bbeta,\lambda}(Y_{ij})\le \E_{\bbeta,\lambda} Y_{ij}\le \frac{\lambda}{n}.
$
\end{proof}

\begin{proof}[Proof of Theorem \ref{thm:dense}. \ref{thm:dense_upper}]
%Throughout the proof we drop the suffix of $\lambda_n,s_n,A_n$, noting that they all depend on $n$.
Let $\phi_n$ be the test defined by
\be
\phi_n=&1\text{ if }\sum_{i=1}^nd_i-\frac{(n-1)\lambda}{2}>K_n\\
=&0\text{ otherwise},
\ee
where $K_n:=\frac{\lambda s}{8}\tanh\Big(\frac{A}{2}\Big)$. We define the sub parameter space
\be
\widetilde{\Xi}(s,A):=\{\beta\in \mathbb{R}_+^n: |S(\beta)|=s, \beta_i= A, i\in S(\beta)\}. \label{eq:sub}
\ee
Since the distribution of $\sumin d_i$ is stochastically increasing in $A$, it suffices to prove that 
\be\label{eq:dense1}
\limsup_{n\rightarrow\infty}\P_{\bbeta=\mathbf{0},\lambda}(\sum_{i=1}^nd_i-\frac{(n-1)\lambda}{2}>K_n)=0,\\
\label{eq:dense2}\limsup_{n\rightarrow\infty}\sup_{\bbeta\in \widetilde{\Xi}(s,A)} \P_{\bbeta,\lambda}(\sum_{i=1}^nd_i- \frac{(n-1)\lambda}{2}>K_n)=1.
\ee
For proving \eqref{eq:dense1}, an application of Chebyshev's inequality along with Lemma \ref{lemma_totdeg_var} gives
\be
\P_{\bbeta=\mathbf{0},\lambda}(\sum_{i=1}^nd_i-\frac{(n-1)\lambda}{2}>K_n)=&\P_{\bbeta=\mathbf{0},\lambda}(\sum_{i=1}^nd_i-\sum_{i=1}^n\E_{\bbeta=0,\lambda}d_i>K_n)
\le \frac{2n\lambda}{K_n^2},
\ee
which goes to $0$ using \eqref{eq:upper_dense} by choice of $K_n$, thus proving \eqref{eq:dense1}. Turning to prove \eqref{eq:dense2}, for any $\bbeta\in \widetilde{\Xi}(s,A)$  we have
\be
\ &\sum_{i=1}^n\E_{\bbeta,\lambda}d_i-\frac{(n-1)\lambda}{2}\\=&\frac{\lambda}{n}s(s-1)\Big(\frac{e^{2A}}{1+e^{2A}}-\frac{1}{2}\Big)+\frac{\lambda}{n}2s(n-s)\Big(\frac{e^{A}}{1+e^{A}}-\frac{1}{2}\Big)\\
=&\frac{\lambda}{2n}s(s-1)\tanh(A)+\frac{\lambda}{n}s(n-s)\tanh\Big(\frac{A}{2}\Big)\\
\ge &\frac{\lambda}{2n}\tanh\Big(\frac{A}{2}\Big)\Big[s(s-1)+s(n-s)\Big]\\ %\text{ [ Since }\tanh(A)\le 2\tanh\Big(\frac{A}{2}\Big)]\\
=&\frac{\lambda s(n-1)}{2n}\tanh\Big(\frac{A}{2}\Big)
\ge \frac{\lambda s}{4}\tanh\Big(\frac{A}{2}\Big).
\ee
This immediately gives
\be
\ & \P_{\bbeta,\lambda}\Big(\sum_{i=1}^n d_i-\frac{(n-1)\lambda}{2}\le K_n\Big)
\\&=\P_{\bbeta,\lambda}\left(\sum_{i=1}^n(d_i-\E_{\bbeta,\lambda}d_i)\le-\sum_{i=1}^n\E_{\bbeta,\lambda}d_i+\frac{(n-1)\lambda}{2}+K_n\right)\\
&\le \P_{\bbeta,\lambda}\left(\sum_{i=1}^n (d_i-\E_{\bbeta,\lambda}d_i)\le-\frac{\lambda s}{4}\tanh\Big(\frac{A}{2}\Big)+K_n\right)\\
&=\P_{\bbeta,\lambda}(\sum_{i=1}^n (d_i-\E_{\bbeta,\lambda}d_i)\le-K_n)
\le \frac{2n\lambda}{K_n^2},
\ee
where the last step uses  Chebyshev's inequality Lemma \ref{lemma_totdeg_var}. This converges to $0$ as before, thus verifying \eqref{eq:dense2}.
\end{proof}

\begin{proof}[Proof of Theorem \ref{thm:dense}. \ref{thm:dense_lower}]
Recall the sub-parameter space $\widetilde{\Xi}(s,A_n)$ from \eqref{eq:sub}. 
Let $\pi(d\bbeta)$ be a prior on $\widetilde{\Xi}(s,A_n)$ which puts mass $\frac{1}{{n\choose s}}$ on each of the configurations in $\widetilde{\Xi}(s,A)$, and let $\Q_{\pi}(.):=\int \P_{\bbeta,\lambda}(.)\pi(d\bbeta)$ denote the marginal distribution of ${\bf Y}$ where $${\bf Y}|\bbeta\sim \P_{\bbeta,\lambda},\quad \bbeta\sim \pi.$$
Then, setting $$L_\pi({\bf Y}):=\frac{\Q_{\pi}({\bf Y})}{\P_{\bbeta=0,\lambda}({\bf Y})}$$ denote the likelihood ratio, it suffices to show that \citep{Ingster4}
\begin{align}
\lim_{n\rightarrow\infty}\E_{\bbeta=\mathbf{0},\lambda}L_\pi^2=1.
\end{align}
To this effect, by a direct calculation we have
\begin{align}\label{eq:dense3}
\E_{\bbeta=\mathbf{0},\lambda}L_\pi^2
=&\frac{1}{{n\choose s}^2}\sum_{S_1,S_2\subset [n]:|S_1|=|S_2|=s}\prod_{1\le i<j\le n}T_{S_1,S_2}^{ij}(A),
\end{align}
where setting $f(A):=\frac{e^A}{1+e^A}$ we define $T_{S_1,S_2}^{ij}(A)$ via the following case by case definition:
\begin{itemize}
\item
{ $i,j\in S_1\cap S_2$}

In this case, set 
\begin{align*}
T_{S_1,S_2}^{ij}(A):=&\frac{2\lambda f(2A)^2}{n}+\frac{\Big(1-\frac{\lambda f(2A)}{n}\Big)^2}{1-\frac{\lambda}{2n}}
\end{align*}
There are ${Z\choose 2}$ such terms, where $Z=|S_1\cap S_2|$.
\\

\item
{ $i\in S_1\cap S_2, j\in (S_1-S_2)\cup (S_2-S_1)$  or vice versa}

In this case, set
\begin{align*}
T_{S_1,S_2}^{ij}(A):=&\frac{2\lambda f(2A)f(A)}{n}+\frac{\Big(1-\frac{\lambda f(2A)}{n}\Big)\Big(1-\frac{\lambda f(A)}{n}\Big)}{1-\frac{\lambda}{2n}}
\end{align*}
There are $2Z(s-Z)$ such terms.
\\

\item
{$i\in S_1-S_2,j\in S_2-S_1$ or vice versa}

In this case, set
\begin{align*}
T_{S_1,S_2}^{ij}(A):=&\frac{2\lambda f(A)^2}{n}+\frac{(1-\frac{\lambda f(A)}{n})^2}{1-\frac{\lambda}{2n}}
\end{align*}
There are $(s-Z)^2$ such terms.
\\

\item
{$i\in S_1\cap S_2,j\in (S_1\cup S_2)^c$ or vice versa}

In this case, set
\begin{align*}
T_{S_1,S_2}^{ij}(A):=&\frac{2\lambda f(A)^2}{n}+\frac{(1-\frac{\lambda f(A)}{n})^2}{1-\frac{\lambda}{2n}}
\end{align*}
There are $Z(n-2s+Z)$ such terms.
\\

\item

For all other cases, set $T_{S_1,S_2}^{ij}(A):=1$.
\end{itemize}

Since $\alpha\ge \frac{1}{2}$, \eqref{eq:lower_dense} 
 implies 
\be
\lim_{n\rightarrow\infty}A\sqrt{\lambda}=0.\label{eq:dense4}
\ee
Using the fact that $f(0)=\frac{1}{2}, f'(0)=\frac{1}{4}$ and all derivatives of $f$ are uniformly bounded, a two term Taylor expansion around $A=0$ gives  the uniform bound $$T_{S_1,S_2}^{ij}(A)\le 1+CA^2\frac{\lambda}{n},\quad \forall A>0,$$
where $C$ is a finite positive constant. Thus, setting
$$N(z):=\{(S_1,S_2)\subset [n]:|S_1|=|S_2|=s,|S_1\cap S_2|=z\}$$
using \eqref{eq:dense3} we have
\begin{align*}
\ & \E_{\bbeta=\mathbf{0},\lambda}L_\pi^2\\ &\le  \frac{1}{{n\choose 2}^2}\sum_{z=0}^s\sum_{(S_1,S_2)\in N(z)}\Big(1+CA^2\frac{\lambda}{n}\Big)^{{z\choose 2}+2z(s-z)+(s-z)^2+z(n-2s+z)}|N(z)|\\
\le &  \frac{1}{{n\choose 2}^2}\sum_{z=0}^s\sum_{(S_1,S_2)\in N(z)} \text{ exp }\Big\{CA^2\frac{\lambda}{n}\Big(\frac{z^2}{2}+s^2+z(n-2s)\Big)\Big\}|N(z)|\\
\le & e^{\frac{3CA^2s^2\lambda}{2n}}\E_Z e^{CA^2\lambda \frac{n-2s}{n}Z},
\end{align*}
where $Z$ is a Hypergeometric distribution with parameters $(n,s,s)$ and $\E_Z$ refers to the expectation w.r.t $Z$. If $2s\ge n$ then we have $\E_Z e^{CA^2\lambda \frac{n-2s}{n}Z}\le 1$, giving
$$\E_{\bbeta=\mathbf{0},\lambda} L_\pi^2\le e^{\frac{3CA^2s^2\lambda}{2n}},$$
the RHS of which converges to $1$ on using \eqref{eq:lower_dense}.
 Thus assume without loss of generality that $2s< n$, in which case $Z$ is stochastically dominated by a Binomial distribution with parameters $\Big(s,\frac{s}{n-s}\Big)$, which gives
\begin{align*}
\E  e^{CA^2\lambda \frac{n-2s}{n}Z}\le \E_Z e^{CA^2\lambda Z}\le & \Big[1+\frac{s}{n-s}\Big(e^{CA^2\lambda}-1\Big)\Big]^s\\
\le &\text{ exp }\Big\{\frac{2s^2}{n}\Big(e^{CA^2\lambda} -1\Big)\Big\}.
\end{align*}
Combining, we have the bound
$$\E_{\bbeta=0,\lambda}L_\pi^2\le \text{ exp }\Big\{\frac{3CA^2s^2\lambda}{2n}+\frac{2s^2}{n}\Big(e^{CA^2\lambda} -1\Big)\Big\},$$
the RHS of which converges to $1$ on using \eqref{eq:dense4} along with \eqref{eq:lower_dense}.
\end{proof}

\subsection{Proof of Theorem \ref{thm:sparse}}

Throughout the proof we drop the suffix of $\lambda_n,s_n,A_n$, noting that they all depend on $n$.

\begin{proof}[Proof of Theorem \ref{thm:sparse} \ref{thm:sparse_hopeless}]
We proceed exactly as in the proof of the lower bound in Theorem \ref{thm:dense}. We use the uniform prior $\pi$ on $\widetilde{\Xi}(s,A)$. 
Then, setting $L_\pi({\bf Y})$ to denote the likelihood ratio, we shall establish that 
\be
\lim_{n\rightarrow\infty}\E_{\bbeta=\mathbf{0},\lambda}L_\pi^2=1.
\ee
Similar to the proof for Theorem \ref{thm:dense},we express
\be
\E_{\bbeta=\mathbf{0},\lambda}L_\pi^2
=&\frac{1}{{n\choose s}^2}\sum_{S_1,S_2\subset [n]:|S_1|=|S_2|=s}\prod_{1\le i<j\le n}T_{S_1,S_2}^{ij}(A),\label{eq:dense3}
\ee
where $T_{S_1,S_2}^{ij}$ are exactly the same as the earlier definition. As before, we set $f(x) = e^x/(1+e^x)$. For $0\leq \mu \leq 1$, and constants $c_1, c_2 >0$, we consider the function 
\begin{align}
h(x) = 4\mu f(c_1 x) f(c_2 x) + \frac{(1- 2 \mu f(c_1x))(1-2 \mu f(c_2 x))}{1-\mu}. \nonumber 
\end{align}
Direct computation yields
\begin{align}
h'(x) &= \frac{4\mu}{1-\mu} \Big[ c_1 f'(c_1 x) f(c_2 x) + c_2 f(c_1 x) f'(c_2 x) \Big]\\& - \frac{2\mu}{1-\mu} \Big[c_1 f'(c_1 x) + c_2 f'(c_2 x) \Big] \geq 0, \nonumber
\end{align}
since $f(x) \geq 1/2$ for $x \geq 0$. Next, we observe that $f(x) \uparrow 1$ as $x \to \infty$. Thus 
\begin{align}
h(x) \leq 4\mu + \frac{(1- 2\mu)^2}{1-\mu}. \nonumber
\end{align}
The expressions derived for $T_{S_1, S_2}^{ij}(A)$ in the proof of the lower bound in Theorem \ref{thm:dense} imply that each term is exactly of the form $h$, with $\mu= \lambda/2n$ and appropriate $c_1,c_2$. Thus, using the upper bound on $h$, we have, 
\be
\E_{\bbeta=\mathbf{0},\lambda}L_\pi^2 &\leq \E_Z \Big[ \Big(\frac{2\lambda}{n} + \frac{(1- \lambda/n)^2}{1- \lambda/2n} \Big)^{{Z \choose 2} + 2Z (s-Z) + Z(n-2s+Z) + (s-Z)^2}\Big]  \nonumber\\
&\leq \E_Z \Big[ \Big(\frac{2\lambda}{n} + \frac{(1- \lambda/n)^2}{1- \lambda/2n} \Big)^{7s^2/2 + n Z}\Big], \label{eq:upperbound_intermediate}
\ee
where $\E_Z[\cdot]$ denotes expectation with respect to $Z$ and $Z$ has a Hypergeometric distribution with parameters $(n,s,s)$ respectively. Next, we note that 
\begin{align}
\frac{2\lambda}{n} + \frac{(1- \lambda/n)^2}{1- \lambda/2n}  = 1 + \frac{\lambda}{2n} + \frac{(\frac{\lambda}{2n})^2}{1- \frac{\lambda}{2n}} \leq 1 + \frac{\lambda}{n} \nonumber
\end{align}
for $n$ sufficiently large as $\lambda \ll \log n$. Thus plugging the bounds back into \eqref{eq:upperbound_intermediate}, we have, for $n$ sufficiently large, 
\be
\E_{\bbeta=\mathbf{0},\lambda}L_\pi^2  \leq  \Big( 1 + \frac{\lambda}{n} \Big)^{7s^2/2} \E_Z\Big[ \Big(1+ \frac{\lambda}{n} \Big)^{nZ}\Big]. \label{eq:upperbound_intermediate1}
\ee
We first observe that $(1+ \lambda/n)^{7s^2/2} \leq \exp(\frac{7s^2\lambda}{2n} ) \to 1 $ as $n \to \infty$, using $s= n^{1-\alpha}$ for some $\alpha>1/2$ and $\lambda \ll \log n$. To bound the second term in \eqref{eq:upperbound_intermediate1}, we note that $(1+\lambda/n) >1$ and therefore
\be
\E_Z \Big[\Big(1+ \frac{\lambda}{n} \Big)^{nZ} \Big] \leq \E_U \Big[\Big(1+ \frac{\lambda}{n} \Big)^{nU} \Big], \label{eq:upperbound_intermediate2}
\ee
where $U\sim \Bin(s, \frac{s}{n-s})$ and $\E_U$ denotes the expectation with respect to $U$. Finally, we have,
\be
\ &\E_U\Big[ \Big(1+ \frac{\lambda}{n} \Big)^{nU}\Big] \\
&= \Big[ 1  + \frac{s}{n-s} \Big( \Big( 1 + \frac{\lambda}{n}\Big)^n -1 \Big) \Big]^s 
\leq \exp\Big(\frac{s^2}{n-s} \Big( \Big( 1 + \frac{\lambda}{n}\Big)^n -1 \Big)  \Big) \nonumber\\
&\leq \exp\Big( \frac{s^2}{n-s} e^{\lambda} \Big) 
\leq \exp\Big( \frac{s^2}{n-s} e^{c \log n}\Big)
\leq \exp(n^{1-2\alpha + c}) \nonumber 
\ee
for any constant $c>0$, arbitrarily small, and $n$ sufficiently large. Thus for any $\alpha>1/2$, we can choose $c$ sufficiently small such that $\alpha > 1/2 + c$. This implies that $\E_U [(1+\lambda/n)^{nU}] \to 1$ as $n\to \infty$. Using \eqref{eq:upperbound_intermediate2} and plugging this bound back into \eqref{eq:upperbound_intermediate1} gives the desired conclusion. 
\end{proof}

\begin{proof}[Proof of Theorem \ref{thm:sparse} \ref{thm:sparse_hopeful} \ref{thm:sparse_upperbound}]
Recall the version of the higher criticism test introduced in Section \ref{section:tests}. We will reject the null hypothesis if 
$HC > \sqrt{\log n}$. 
%\be 
%HC:=\sup\left\{GHC(t):=\frac{HC(t)}{\sqrt{Var_{\boldsymbol{\beta}=\mathbf{0},\lambda}\left(HC(t)\right)}}, t\in \{\sqrt{2r\log{n}}:r\in (0,5)\}\cap \mathbb{N}\right\}. 
%\ee
By virtue of centering and scaling of individual $HC(t)$ under the null, we have by union bound and Chebyshev's Inequality,
\be 
\ &\Pzero\left(HC\geq \sqrt{\log{n}}\right)\\ &\leq \sum_t \Pzero \left(GHC(t) > \sqrt{\log {n}} \right)\leq \frac{\sqrt{10\log{n}}}{\log{n}}\rightarrow 0 \quad \text{as}\ n\rightarrow \infty.
\ee
This controls the Type I error of this test. It remains to control the Type II error. We will establish as usual that the non-centrality parameter under the alternative beats the null and the alternative variances of the statistic. We consider alternatives as follows. Let $\Pbeta$ be such that $\beta_i=A$ for $i\in S$ and $\beta_i=0$ otherwise, where $A=\sqrt{C^*\frac{\log{n}}{\lambda}}$ with $16(1-\theta)\geq C^*>\csparse(\alpha)$, $\theta=\lim \frac{\lambda}{2n}$, $|S|=s=n^{1-\alpha}$, $\alpha \in (1/2,1)$. The case of higher signals can be handled by standard monotonicity arguments and are therefore omitted. 
The following Lemma studies the behavior of this statistic under this class of alternatives. 
\begin{lemma}\label{lemma:power_hc} Let $t=\lfloor\sqrt{2r\log{n}}\rfloor$ with $r=\min\left\{1,\frac{C^*}{4(1-\theta)}\right\}$. Then
\begin{enumerate} 
\item[(a)] $\Ebeta\left(GHC(t)\right)\gg \sqrt{\log{n}}.$ \label{lemma:power_hc_a}
\item[(b)]  $\Ebeta^2\left(GHC(t)\right)\gg \mathrm{Var}_{\boldsymbol{\beta},\lambda}\left(GHC(t)\right).$ \label{lemma:power_hc_b}
\end{enumerate}
\end{lemma}

The Type II error of the HC statistic may be controlled immediately using Lemma \ref{lemma:power_hc}. This is straightforward--- however, we include a proof for the sake of completeness. For any alternative considered above, we have, using Chebychev's inequality and Lemma \ref{lemma:power_hc}, 
\be
\ &\Pbeta[HC > \sqrt{\log n}]\\ &\geq \Pbeta [GHC(t) \geq \sqrt{\log n}] 
\geq 1 - \frac{ \mathrm{Var}_{\boldsymbol{\beta},\lambda}\left(GHC(t)\right)}{(\Ebeta\left(GHC(t)\right) - \sqrt{\log n})^2} \to 1
\ee
as $n \to \infty$. This completes the proof, modulo that of Lemma \ref{lemma:power_hc}. 
\end{proof}

We describe the proof of Lemma \ref{lemma:power_hc} next. This necessitates a detailed understanding of the mean and variance of the $HC(t)$ statistics introduced in Section \ref{section:tests}. Due to centering, $HC(t)$ has mean $0$ under the null hypothesis. Our next proposition estimates the variances of the $HC(t)$ statistics under the null and the class of alternatives introduced above. We also lower bound the expectation of the $HC(t)$ statistics under the alternative. This is the most technical result of this paper and we defer the proof to the Appendix \ref{sec:proof_lemma_hc}. 

\begin{prop}
\label{lemma:hc_main}
Fix $\theta = \displaystyle\lim_{n\to \infty} \frac{\lambda}{2n}$. For $t = \lfloor\sqrt{2 r \log n}\rfloor$ with $r > \frac{C^*}{16(1-\theta)}$, we have,
\be
\lim_{n\to \infty} \frac{\log \Vzero(HC(t)) }{\log n} &= 1- r, \label{eq:null_var} \\ 
\lim_{n \to \infty} \frac{\log \Ebeta\left(HC(t)\right)}{\log n}  &\geq 1- \alpha -\frac{1}{2} \left(\sqrt{2r}-\sqrt{ \frac{C^*}{8(1-\theta)}}\right)^2 , \label{eq:alt_exp}\\
\lim_{n \to \infty}  \frac{ \log \mathrm{Var}_{\boldsymbol{\beta},\lambda}\left(HC(t)\right) }{\log n}&= \max\left\{ 1-\alpha -\frac{1}{2} \left(\sqrt{2r}-\sqrt{\frac{C^*}{8(1-\theta)}}\right)^2, 1- r \right\}. \\\label{eq:alt_var}
\ee
\end{prop}

\begin{proof}[Proof of Lemma \ref{lemma:power_hc}]
We first look at the proof of Part(a). 
%\subsubsection{Proof of (a)} 
Using \eqref{eq:null_var} and \eqref{eq:alt_exp} we have,
\be
&\Ebeta\left(GHC(t)\right)\geq n^{(f_0(r)+o(1))}, \\
& f_0 (r) = \frac{1}{2} - \alpha - \frac{r}{2}  - \frac{C^*}{ 16(1-\theta)} + { \sqrt{\frac{rC^*}{4(1-\theta)} } } .
\ee
%
%
%  f_0(r)= \frac{1}{2} - \alpha - \frac{r}{2} +  \sqrt{\frac{rC^*}{4(1-\theta)} }  - \frac{C^*}{ 16(1-\theta)} .
%
%\be 
%&\Ebeta\left(GHC(t)\right)\geq n^{(f_0(r)+o(1))}, \\
%&f_0(r)=\frac{1}{2}-\alpha- \frac{r}{2}- \frac{C^*}{16(1-\theta)}+ \sqrt{\frac{C^*}{4(1-\theta)}}.
% \frac{sn^{-\left(\sqrt{2r}-\sqrt{C^*/8(1-\theta)}\right)^2/2+o(1)}\left(1+o(1)\right)}{n^{1/2-(\sqrt{2r})^2/4+o(1)}}\\
%& \gtrsim n^{1-\alpha-r-C^*/16(1-\theta)+\sqrt{rC^*/4(1-\theta)}-1/2+r/2+o(1)}(1+o(1))\\
%\ee

%\be 
%&=1-\beta-1-C^*/16(1-\theta)+\sqrt{rC^*/4(1-\theta)}.
%\ee
Thus it suffices to show that    $f_0(r) >0 $ for $r=\min\{1, \frac{C^*}{4(1-\theta)}\}$. Now, $C^* > \csparse(\alpha)$ implies that 
$C^* > 4(1-\theta)$ for $\alpha \geq 3/4$. Thus in this case, $r = 1$. Therefore, we have, for $\alpha \geq 3/4$,  
\be 
f_0 (r) = f_0(1)
%&=\frac{1}{2}-\alpha-\frac{1}{2}- \frac{C^*}{16(1-\theta)}+\sqrt{ \frac{C^*}{4(1-\theta)}}\\
&=1-\alpha-\left(1-{ \sqrt { \frac{C^*}{16(1-\theta)}  } } \right)^2>0,
%&=1-\alpha-1-C^*/16(1-\theta)+\sqrt{rC^*/4(1-\theta)}.
\ee
as $16(1-\theta)\geq C^*> \csparse(\alpha)$. 
%Therefore, if $16(1-\theta)>C^*>16(1-\theta)(1-\sqrt{1-\alpha})^2$ we are done. Consider now the case when $16(1-\theta)(\alpha-1/2)<C^* \leq 16(1-\theta)(1-\sqrt{1-\alpha})^2$ which can happen only when $\alpha<3/4$. 
%
For $\alpha< 3/4$, we consider two cases. First, consider the case where $C^* > 4(1-\theta)$. As $\alpha< 3/4$, this 
implies that $C^* > 16(1-\theta) ( 1- \sqrt{1-\alpha})^2$. In this case, $r=1$ and the same argument outlined above still goes through. Finally, consider the case $C^* < 4(1-\theta)$. In this case,
 $r= \frac{C^*}{4(1-\theta)}\leq 1$. We have, 
\be 
f_0(r) = f_0\Big( \frac{C^*}{4(1-\theta)} \Big)=1/2-\alpha+ \frac{ C^* } {16(1-\theta) }>0, \quad \text{since} \ C^*> \csparse(\alpha).
\ee

%\subsubsection{Proof of (b)} 
Next, we turn to the proof of Part (b). 
We use \eqref{eq:alt_exp} and \eqref{eq:alt_var} to get, 
\be 
 \frac{\Ebeta^2\left(GHC(t)\right)}{\mathrm{Var}_{\boldsymbol{\beta},\lambda}\left(GHC(t)\right)} &=\frac{\Ebeta^2\left(HC(t)\right)}{\mathrm{Var}_{\boldsymbol{\beta},\lambda}\left(HC(t)\right)}
%
%&\gtrsim \frac{s^2n^{-\left(\sqrt{2r}-\sqrt{C^*/8(1-\theta)}\right)^2+o(1)}\left(1+o(1)\right)}{n^{1-\alpha-\left(\sqrt{2r}-\sqrt{C^*/8(1-\theta)}\right)^2/2+o(1)}+n^{1-(\sqrt{r})^2+o(1)}}\\
\geq \min\left\{n^{f_1(r) + o(1)} ,n^{f_2(r) + o(1)}\right\}, \\
f_1(r)&=1-\alpha- \frac{1}{2} \left(\sqrt{2r}-\sqrt{ \frac{ C^*}{ 8(1-\theta) } }\right)^2 \\
f_2(r)&=1-2\alpha-\left(\sqrt{2r}-\sqrt{ \frac{C^*}{8(1-\theta)} }\right)^2+ r.
\ee
%and
%\be 
%
%\ee
Therefore, it suffices to show that both $f_1(r)$ and $f_2(r)$ are strictly positive for $r=\min\{1, \frac{C^*}{4(1-\theta)} \}$. To this end, we again note, as in Part (a), that $r=1$ for $\alpha\geq 3/4$.  In this case, 
\be
f_1(r)= f_1(1)&=1-\alpha-\left(1-\sqrt{ \frac{C^*}{16(1-\theta)}}\right)^2>0
%=1-\alpha-1-C^*/16(1-\theta)+\sqrt{C^*/4(1-\theta)}\\&
\ee
as $\csparse(\alpha)<C^*\leq 16(1-\theta)$. Also,
$f_2(r)= f_2(1)=2f_1(1)>0.$
%

%Therefore, if $16(1-\theta)>C^*>16(1-\theta)(1-\sqrt{1-\alpha})^2$ it suffices to take $r=1$. Therefore consider the case when $16(1-\theta)(\alpha-1/2)<C^* \leq 16(1-\theta)(1-\sqrt{1-\alpha})^2$ which can happen only when $\alpha<3/4$.
%
It remains to establish the desired proposition for $\alpha < 3/4$. Again, we split the argument into two cases. If $C^*> 4(1-\theta)$, $\alpha<3/4$ implies that $C^*> 16(1-\theta) (1- \sqrt{1-\alpha})^2$. In this case, $r=1$ and the argument outlined above still goes through. Finally, we consider the case $C^* < 4(1-\theta)$. 
 In this case, $r= \frac{C^*}{4(1-\theta)}$. We have,
\be 
f_2(r) = f_2 \left(\frac{C^*}{4(1-\theta)} \right)
%&=1-2\alpha-C^*/4(1-\theta)-C^*/8(1-\theta)+C^*/2(1-\theta)\\
%&=1-2\alpha-C^*/8(1-\theta)
&=\frac{1}{8(1-\theta)}\left(C^*- 16(1-\theta) \Big(\alpha- \frac{1}{2}\Big) \right)>0.
\ee
Finally, we prove that 
\be 
f_1(r)= f_1 \left(\frac{C^*}{4(1-\theta)} \right)=1-\alpha- \frac{C^*}{16(1-\theta)}>0.
\ee
The validity of the above display is established by contradiction. Suppose, if possible, $f_1(r) \leq 0$. Then one must have $C^*\geq 16(1-\theta)(1-\alpha) \geq 4(1-\theta)$ as $\alpha\leq 3/4$. This is a contradiction to the assumption that $C^* < 4(1-\theta)$. This completes the proof. 
%
%But we know that $C^* \leq 16(1-\theta)(1-\sqrt{1-\alpha})^2$. This forces $(1-\alpha)\leq (1-\sqrt{1-\alpha})^2$, which in turn implies that $\alpha\geq 3/4$ and thereby giving us our required contradiction.
 %If $C^*>16(1-\theta)(1-\sqrt{1-\alpha})^2$, then we can still consider $f_1(1)$ and it is positive as before. Otherwise assume $C^*\leq 16(1-\theta)(1-\sqrt{1-\alpha})^2$. Then consider, 
%\be 
%
%\ee
%Therefore
%Now by Lemma \ref{lemma:binomial_collection}
%\be 
%\ & I+II+III+IV+V+VI \\
%&\leq s\left( n^{-\left(\sqrt{2r}-\sqrt{C^*/8(1-\theta)}\right)^2/2+o(1)}+2n^{-M}\right)+n\left(n^{-(\sqrt{r})^2+o(1)}+2n^{-M}\right)\\
%&+\frac{s^2}{n}\lambda f(2A)\left(1-\frac{\lambda}{n}f(2A)\right)\left(\frac{n^{-(\sqrt{r}-\sqrt{C^*/16(1-\theta)})^2+o(1)}}{\sqrt{\pi \lambda (1-\theta)(1+o(1)) }}+2n^{-M}\right)^2\\
%&+\frac{n^2}{n}\lambda f(0)\left(1-\frac{\lambda}{n}f(0)\right)\left(\frac{n^{-(\sqrt{r})^2+o(1)}}{\sqrt{\pi \lambda (1-\theta)(1+o(1)) }}+2n^{-M}\right)^2\\
%&+\frac{2ns}{n}\lambda f(A)\left(1-\frac{\lambda}{n}f(A)\right)\left(\frac{n^{-(\sqrt{r}-\sqrt{C^*/16(1-\theta)})^2+o(1)}}{\sqrt{\pi \lambda (1-\theta)(1+o(1)) }}+2n^{-M}\right)\left(\frac{n^{-(\sqrt{r})^2+o(1)}}{\sqrt{\pi \lambda (1-\theta)(1+o(1)) }}+2n^{-M}\right).
%\ee
%
%
 %
 \end{proof}
%\end{proof} 

\begin{proof}[Proof of Theorem \ref{thm:sparse} \ref{thm:sparse_hopeful} \ref{thm:sparse_lowerbound}]
%As in the proof of Theorem \ref{thm:dense}, let $\pi(d\beta)$ be a prior on $\widetilde{\Xi}(s,A_n)$ which puts mass $\frac{1}{{n\choose s}}$ on each of the configurations in $\widetilde{\Xi}(s,A)$, where where $A=\sqrt{C^*\frac{\log{n}}{\lambda}}$ with $ C^*<C(\alpha)$, $\theta=\lim \frac{\lambda}{2n}$, $s=n^{1-\alpha}$, $\alpha \in (1/2,1)$. Let $\Q_{\pi}(.):=\int \P_{\beta,\lambda}(.)\pi(d\beta)$ denote the marginal distribution of ${\bf Y}$ where $$({\bf Y}|\beta)\sim \P_{\beta,\lambda},\quad \beta\sim \pi,$$
%and let $$L_\pi({\bf Y}):=\frac{\Q_{\pi}({\bf Y})}{\P_{\beta=0,\lambda}({\bf Y})}$$ 
%to denote the likelihood ratio. 
Recall the definition of the sub-parameter space $\widetilde{\Xi}(s,A_n)$ from the proof of Theorem \ref{thm:dense}. 
As in the proof of Theorem \ref{thm:dense}, let $\pi(d\bbeta)$ be a prior on $\widetilde{\Xi}(s,A_n)$ which puts mass $\frac{1}{{n\choose s}}$ on each of the configurations in $\widetilde{\Xi}(s,A)$, where $A=\sqrt{C^*\frac{\log{n}}{\lambda}}$ with $ C^*<\csparse(\alpha)$, $\theta=\lim \frac{\lambda}{2n}$, $s=n^{1-\alpha}$, $\alpha \in (1/2,1)$. Let $\Q_{\pi}(.):=\int \P_{\bbeta,\lambda}(.)\pi(d\bbeta)$ denote the marginal distribution of ${\bf Y}$ where $${\bf Y}|\bbeta \sim \P_{\bbeta,\lambda},\quad \bbeta\sim \pi,$$
and let $$L_\pi({\bf Y}):=\frac{\Q_{\pi}({\bf Y})}{\P_{\bbeta=0,\lambda}({\bf Y})}$$ 
 denote the likelihood ratio.  We introduce some notation for ease of exposition.
 For any function $h:2^{[n]}\rightarrow \mathbb{R}$, we denote
\be 
\E_S h(S)=\frac{1}{{n \choose s}}\sum\limits_{S \subseteq [n], \atop |S|=s}h(S).
\ee
Further, we set $\Sigma(S,S)= \sum_{i<j , i,j \in S} Y_{ij}$ and $\Sigma(S,S^c)= \sum_{i \in S, j \in S^c} Y_{ij}$. 
For $\bbeta \in \widetilde{\Xi}(s,A)$,  $S:=\mathrm{Support}(\bbeta)\subseteq [n]$, we define
\be 
 L_S&:=\prod\limits_{i<j}\left(\frac{2e^{\beta_i+\beta_j}}{1+e^{\beta_i+\beta_j}}\right)^{Y_{ij}}\left(\frac{1-\frac{\lambda}{n}\frac{e^{\beta_i+\beta_j}}{1+e^{\beta_i+\beta_j}}}{1-\frac{\lambda}{2n}}\right)^{1-Y_{ij}}\\
%&=\prod\limits_{i \in S}\left(\frac{2e^{2A}}{1+e^{2A}}\right)^{\sum\limits_{j \in S}Y_{ij}}\left(\frac{1-\frac{\lambda}{n}\frac{e^{2A}}{1+e^{2A}}}{1-\frac{\lambda}{2n}}\right)^{(s-1)-\sum\limits_{j \in S}Y_{ij}}
%	\left(\frac{2e^{A}}{1+e^{A}}\right)^{\sum\limits_{j \in S^c}Y_{ij}}\left(\frac{1-\frac{\lambda}{n}\frac{e^{A}}{1+e^{A}}}{1-\frac{\lambda}{2n}}\right)^{(n-s)-\sum\limits_{j \in S^c}Y_{ij}}\\
	&=\left(\frac{2e^{2A}}{1+e^{2A}}\right)^{\Sigma(S,S)}\left(\frac{1-\frac{\lambda}{n}\frac{e^{2A}}{1+e^{2A}}}{1-\frac{\lambda}{2n}}\right)^{{s \choose 2}-\Sigma(S,S)}\\ & \times \left(\frac{2e^{A}}{1+e^{A}}\right)^{\Sigma(S,S^c)}\left(\frac{1-\frac{\lambda}{n}\frac{e^{A}}{1+e^{A}}}{1-\frac{\lambda}{2n}}\right)^{s(n-s)-\Sigma(S,S^c)}.
\ee

Armed with this notation, we note that 
$L_\pi({\bf Y})=\E_S(L_S)$. 
Next, for $S \subseteq [n]$ and $i \in S$, we define the event 
\be 
\Gamma_{S,i}:=\left\{\frac{\sum\limits_{j \in S^c}Y_{ij}-(n-s)\frac{\lambda}{2n}}{\sqrt{(n-s)\frac{\lambda}{2n}\left(1-\frac{\lambda}{2n}\right)}}\leq \sqrt{2\log{n}}\right\}.
\ee
Subsequently, for any $S\subseteq [n]$, we introduce the truncation event,
\be 
\Gamma_S:=\bigcap_{i\in S}\Gamma_{S,i}. 
\ee

Let
$\tilde{L}:=\frac{1}{{n \choose s}}\sum\limits_{S \subseteq [n], \atop |S|=s}L_S \one_{\Gamma_S}=\E_SL_S \one_{\Gamma_S}$.
The thesis follows provided we establish the following (see for example \cite{arias2013community}). 
	\be 
	\Ezero(\tilde{L})=1+o(1), \label{eqn:firstmoment_toshow}\\
	\Ezero\left(\tilde{L}\right)^2=1+o(1). \label{eqn:secondmoment_toshow}
	\ee
	
\eqref{eqn:firstmoment_toshow} is established in Section \ref{subsec:proof1} while \eqref{eqn:secondmoment_toshow} is established in Section \ref{subsec:proof2}. This completes the proof. 
\subsection{Proof of \eqref{eqn:firstmoment_toshow}}
\label{subsec:proof1}
\be 
\Ezero\left(\tilde{L}\right)&=\Ezero\E_S L_S\one_{\Gamma_S}=\E_S\Ezero L_S\one_{\Gamma_S}=1-\E_S\Ezero L_S\one_{\Gamma^c_{S}}.
\ee
Therefore, it suffices to show that
$\E_S\Ezero L_S\one_{\Gamma^c_{S}}=o(1)$.
Now,
\be 
\Ezero L_S\one_{\Gamma^c_{S}}\leq \sum\limits_{i \in S}\Ezero L_S\one_{\Gamma^c_{S,i}}, \label{eqn:firstmoment_firstsimplification}
\ee
and 
\be 
\Ezero L_S\one_{\Gamma^c_{S,i}}&=\Ezero\left(	\left(\frac{2e^{A}}{1+e^{A}}\right)^{\sum\limits_{j \in S^c}Y_{ij}}\left(\frac{1-\frac{\lambda}{n}\frac{e^{A}}{1+e^{A}}}{1-\frac{\lambda}{2n}}\right)^{(n-s)-\sum\limits_{j \in S^c}Y_{ij}}\one_{\Gamma^c_{S,i}}\right)\\
&=\P\left(\frac{\Bin\left(n-s,\frac{\lambda}{n}\frac{e^A}{1+e^A}\right)-\frac{(n-s)\lambda}{2n}}{\sqrt{(n-s)\frac{\lambda}{2n}\left(1-\frac{\lambda}{2n}\right)}}>\sqrt{2\log{n}}\right).
\ee
The last equality in the above display follows by noting that $$\left(\frac{2e^{A}}{1+e^{A}}\right)^{\sum\limits_{j \in S^c}Y_{ij}}\left(\frac{1-\frac{\lambda}{n}\frac{e^{A}}{1+e^{A}}}{1-\frac{\lambda}{2n}}\right)^{(n-s)-\sum\limits_{j \in S^c}Y_{ij}}$$ is the Radon-Nikodym derivative of $\Bin\left(n-s,\frac{\lambda}{n}\frac{e^A}{1+e^A}\right)$ with respect to $\Bin\left(n-s,\frac{\lambda}{2n}\right)$ which is in turn the distribution of $\sum\limits_{j \in S^c}Y_{ij}$ under $\Pzero$. Denoting $W\sim \Bin\left(n-s,\frac{\lambda}{n}\frac{e^A}{1+e^A}\right)$, we have,
\be 
\Ezero L_S\one_{\Gamma^c_{S,i}}
%&=\P\left(\frac{W-\frac{(n-s)\lambda}{2n}}{\sqrt{(n-s)\frac{\lambda}{2n}\left(1-\frac{\lambda}{2n}\right)}}>\sqrt{2\log{n}}\right)\\
&=\P\left(\frac{W-\frac{(n-s)\lambda}{n}\frac{e^A}{1+e^A}}{\sqrt{(n-s)\frac{\lambda}{n}\frac{e^A}{1+e^A}\left(1-\frac{\lambda}{n}\frac{e^A}{1+e^A}\right)}}>\eta_n\right),
\ee
where
\be 
\eta_n=\frac{(n-s)\frac{\lambda}{n}\left(\frac{1}{2}-\frac{e^A}{1+e^A}\right)+\sqrt{2\log{n} {(n-s)\frac{\lambda}{2n}\left(1-\frac{\lambda}{2n}\right)} }}{\sqrt{(n-s)\frac{\lambda}{n}\frac{e^A}{1+e^A}\left(1-\frac{\lambda}{n}\frac{e^A}{1+e^A}\right)}}.
\ee
Now noting that $\frac{e^A}{1+e^A}=\frac{1}{2}+\frac{A}{4}+O(A^2)$, it is easy to see that
\be 
\eta_n=\sqrt{2\log{n}}\left(1-\sqrt{\frac{C^*}{16(1-\theta)}}+a_n\right),
\ee
where $a_n\rightarrow 0$ as $n \rightarrow \infty$, and the sequence is independent of specific $S\subseteq [n]$ with $|S|=s$.
Since $\lambda \gg \log{n}$, we have using part Lemma \ref{lemma:binomial_collection} Part (b, ii),
\be 
\ &  \Ezero L_S\one_{\Gamma^c_{S,i}} \\&=\P\Bigg(\frac{W-\frac{(n-s)\lambda}{n}\frac{e^A}{1+e^A}}{\sqrt{(n-s)\frac{\lambda}{n}\frac{e^A}{1+e^A}\left(1-\frac{\lambda}{n}\frac{e^A}{1+e^A}\right)}}>\sqrt{2\log{n}}\Bigg(1-\sqrt{\frac{C^*}{16(1-\theta)}}+a_n\Bigg)\Bigg)\\& \leq n^{-\left(1-\sqrt{\frac{C^*}{16(1-\theta)}}\right)^2+o(1)}. \label{eqn:firstmoment_secondsimplification}
\ee
Therefore combining \eqref{eqn:firstmoment_firstsimplification} and \eqref{eqn:firstmoment_secondsimplification} along with the independence of $a_n$ from specific $S$'s, we have
\be 
\E_S\Ezero L_S\one_{\Gamma^c_{S}}& \leq\E_S n^{1-\alpha-\left(1-\sqrt{\frac{C^*}{16(1-\theta)}}\right)^2+o(1)}\\&=n^{1-\alpha-\left(1-\sqrt{\frac{C^*}{16(1-\theta)}}\right)^2+o(1)}=o(1),
\ee
since $1-\alpha-\left(1-\sqrt{\frac{C^*}{16(1-\theta)}}\right)^2<0$ whenever $ C^*<\csparse(\alpha)$. The completes the verification of \eqref{eqn:firstmoment_toshow}.
\subsection{Proof of \eqref{eqn:secondmoment_toshow}}
\label{subsec:proof2}

For any function $h:2^{[n]}\times 2^{[n]}\rightarrow \mathbb{R}$, define
\be 
\E_{S_1,S_2} h(S_1,S_2)=\frac{1}{{n \choose s}^2}\sum\limits_{(S_1,S_2) \subseteq [n]\times [n], \atop |S_1|=|S_2|=s}h(S_1,S_2).
\ee
This allows us to express the second moment of the truncated likelihood ratio as follows. 
\be 
\Ezero\left(\tilde{L}\right)^2&=\Ezero\left(\E_SL_S \one_{\Gamma_S}\right)^2\\
%&=\Ezero\left(\E_{S_1,S_2}\left(L_{S_1}L_{S_2}\one_{\Gamma_{S_1}}\one_{\Gamma_{S_2}}\right)\right)\\
%&=\Ezero\left(\E_{S_1,S_2}\left(L_{S_1}L_{S_2}\one_{\Gamma_{S_1\cap \Gamma_{S_2}}}\right)\right)\\
&=\E_{S_1,S_2}\left(\Ezero\left(L_{S_1}L_{S_2}\one_{\Gamma_{S_1}\cap \Gamma_{S_2}}\right)\right).
\ee
Before proceeding, we introduce some notations. For any set $T \subseteq [n]$, the set $\{(i,j)\in T\}$ denotes all pairs $i<j$ such that both $i$ and $j$ are in $T$. Recall that for any $T\subseteq [n]$, $\Sigma(T,T)= \sum_{(i,j) \in T} Y_{ij}$ while for any two sets 
$S, T \subseteq [n]$, $\Sigma(S, T) = \sum_{i \in S, j \in T} Y_{ij}$. For ease of notation, we will use $\Sigma(S) = \Sigma(S,S)$ if there is no scope for confusion. For $S_1, S_2$ subsets of $[n]$, let $Z=|S_1 \cap S_2|$. Now if $S_1:=\mathrm{Support}(\bbeta)\subseteq [n]$ and  $S_2:=\mathrm{Support}(\bbeta')\subseteq [n]$ for $\bbeta$, $\bbeta' \in \widetilde{\Xi}(s,A)$, we have  
\be 
 L_{S_1}L_{S_2} 
&=\prod\limits_{i<j}\left(\frac{4e^{\beta_i+\beta_j+\beta_i'+\beta_j'}}{(1+e^{\beta_i+\beta_j})(1+e^{\beta_i'+\beta_j'})}\right)^{Y_{ij}}\\ &\times\left(\left(\frac{1-\frac{\lambda}{n}\frac{e^{\beta_i+\beta_j}}{1+e^{\beta_i+\beta_j}}}{1-\frac{\lambda}{2n}}\right)\left(\frac{1-\frac{\lambda}{n}\frac{e^{\beta_i'+\beta_j'}}{1+e^{\beta_i'+\beta_j'}}}{1-\frac{\lambda}{2n}}\right)\right)^{1-Y_{ij}}  . \label{eqn:second_moment_mastereqn}
\ee 
Further, for $i \in S_1 \cap S_2$, define 
\be 
\C_i=\left\{\sum\limits_{j \in (S_1 \cup S_2)^c}Y_{ij}\leq (n-s)\frac{\lambda}{2n}+\sqrt{ 2(n-s)\log{n}\frac{\lambda}{2n}\left(1-\frac{\lambda}{2n}\right)}\right\}. \\ \label{eqn:Ci_sets}
\ee
Finally, we set 
$
\C:=\bigcap_{i \in S_1 \cap S_2}\C_i. 
$
%\label{eqn:C_sets}
%
%
Thus, we have, 
\be 
\ & \Gamma_{S_1}\cap \Gamma_{S_2}\\&=\left\{\begin{array}{c}\sum\limits_{j \in S_1^c}Y_{ij}\leq (n-s)\frac{\lambda}{2n}+\sqrt{2\log{n}}\sqrt{(n-s)\frac{\lambda}{2n}\left(1-\frac{\lambda}{2n}\right)}, \forall i \in S_1,\\
\sum\limits_{j \in S_2^c}Y_{ij}\leq (n-s)\frac{\lambda}{2n}+\sqrt{2\log{n}}\sqrt{(n-s)\frac{\lambda}{2n}\left(1-\frac{\lambda}{2n}\right)}, \forall i \in S_2\end{array}\right\}\\
& \subseteq \left\{\sum\limits_{j \in S_1^c \cap S_2^c}Y_{ij}\leq (n-s)\frac{\lambda}{2n}+\sqrt{2\log{n}}\sqrt{(n-s)\frac{\lambda}{2n}\left(1-\frac{\lambda}{2n}\right)}, \forall i \in S_1\cap S_2 \right\}\\&=\C. 
\ee 
This implies that $ \Ezero\left(L_{S_1}L_{S_2}\one_{\Gamma_{S_1}\cap \Gamma_{S_2}}\right)
\leq \Ezero\left(L_{S_1}L_{S_2}\one_{\C}\right)$. We note that the event $\C$ is a function of the edges $\{ Y_{ij}: i \in S_1 \cap S_2, j \in (S_1 \cup S_2)^c \}$. Further, it is easy to see from \eqref{eqn:second_moment_mastereqn} that 
\be
L_{S_1} L_{S_2}& =   \Upsilon \left(\frac{4e^{2A}}{(1+e^A)^2}\right)^{\Sigma(S_1 \cap S_2, (S_1 \cup S_2)^c)}\\ &\times\left(\frac{1-\frac{\lambda}{n}\frac{e^A}{1+e^A}}{(1-\frac{\lambda}{2n})}\right)^{2 \big(Z(n-2s+Z)-\Sigma(S_1 \cap S_2, (S_1 \cup S_2)^c) \big)},  \label{eqn:second_moment_intermediate} 
\ee
where $\Upsilon$ is independent of $\{ Y_{ij}: i \in S_1 \cap S_2, j \in (S_1 \cup S_2)^c \}$. We next make the following observation. 

\begin{lemma}
\label{lemma:upsilon_value}
$\Ezero(\Upsilon)=1 + o(1)$, uniformly over $Z$ . 
\end{lemma}
The proof is similar to that of the lower bound in Theorem \ref{thm:dense} and thus will be deferred to Appendix \ref{appendix:section_upsilon}. 
Therefore, we have, setting $T_i:=\sum\limits_{j \in (S_1 \cup S_2)^c}Y_{ij}$ for $i \in S_1 \cap S_2$, using Lemma \ref{lemma:upsilon_value} and \eqref{eqn:second_moment_intermediate} along with the independence of the collections $\{ Y_{ij} : j \in (S_1 \cup S_2)^c \}$, $ i \in S_1 \cap S_2$, 
% using \eqref{eqn:second_moment_mastereqn}, \eqref{eqn:C_sets}, and \eqref{eqn:Ci_sets} one has for any $i \in S_1 \cap S_2$
\be 
\ & \Ezero\left(L_{S_1}L_{S_2}\one_{\Gamma_{S_1}\cap \Gamma_{S_2}}\right)\leq \Ezero\left(L_{S_1}L_{S_2}\one_{\C}\right)\\
%&=\Ezero\left(\left(\frac{4e^{2A}}{(1+e^A)^2}\right)^{\Sigma(S_1 \cap S_2, (S_1 \cup S_2)^c)}\left(\frac{1-\frac{\lambda}{n}\frac{e^A}{1+e^A}}{(1-\frac{\lambda}{2n})}\right)^{2\big(Z(n-2s+Z)- \Sigma(S_1 \cap S_2, (S_1 \cup S_2)^c)\big)}\one_{\C}\right)\\
&= (1+ o(1)) \left\{\Ezero\left(\left(\frac{4e^{2A}}{(1+e^A)^2}\right)^{T_i }\left(\frac{1-\frac{\lambda}{n}\frac{e^A}{1+e^A}}{(1-\frac{\lambda}{2n})}\right)^{2\left((n-2s+Z)-T_i\right)}\one_{\C_i}\right)\right\}^Z\\
&=(1+ o(1)) g(A)^{Z(n-2s+Z)}\left(\Ezero\left(f(A)^{T_i}g(A)^{-T_i}\one_{\C_i}\right)\right)^Z, \label{eqn:changeofmeasure_todo}
\ee
where we define 
\be 
f(x):=\frac{4e^{2x}}{(1+e^x)^2}, \,\,\,\,   \Big(1-\frac{\lambda}{2n} \Big)^2g(x):=\Big(1-\frac{\lambda}{n}\frac{e^x}{1+e^x}\Big)^2 . 
\ee

Applying Lemma \ref{lemma:binomial_changeof_measure} with $\psi(A,\lambda,n) = \frac{f(A)}{g(A)} \frac{\lambda}{2n} + 1 - \frac{\lambda}{2n}$,  
%in context of \eqref{eqn:changeofmeasure_todo} we note that $T_i \sim \Bin\left(n-2s+Z,\frac{\lambda}{2n}\right)$ under $\Pzero$. Therefore, 
we immediately have
\be 
\ &\Ezero\left(f(A)^{T_i}g(A)^{-T_i}\one_{\C_i}\right)\\
&=\psi(A,\lambda,n)^{n-2s+Z} \P\left(X\leq (n-s)\frac{\lambda}{2n}+\sqrt{2\log{n}}\sqrt{(n-s)\frac{\lambda}{2n}\left(1-\frac{\lambda}{2n}\right)}\right),\\
 \label{eqn:changeofmeasure_inyourface}
\ee
where $X \sim \Bin\left(n-2s+Z,\frac{f(A)\lambda}{(2n) g(A)\psi(A,\lambda,n)}\right)$. Combining \eqref{eqn:changeofmeasure_todo} and \eqref{eqn:changeofmeasure_inyourface}, we have, 
\be 
\ &\Ezero\left(L_{S_1}L_{S_2}\one_{\Gamma_{S_1}\cap \Gamma_{S_2}}\right)\leq
(1+ o(1))\times  \\ & \left[ \left\{g(A)\psi(A,\lambda,n)\right\}^{n-2s+Z}\P\left(X\leq (n-s)\frac{\lambda}{2n}+\sqrt{2\log{n}}\sqrt{(n-s)\frac{\lambda}{2n}\left(1-\frac{\lambda}{2n}\right)}\right)\right]^Z. \\ \label{eqn:secondmoment_simplified}
\ee
We will bound the RHS of \eqref{eqn:secondmoment_simplified} to complete the proof. 
First, using Taylor expansion, we have, 
\be 
f^*(x)=g(x)\psi(x,\lambda,n) 
%=\frac{2\lambda}{n}\frac{e^{2x}}{(1+e^x)^2}+\frac{\left(1-\frac{\lambda}{n}\frac{e^x}{1+e^x}\right)^2}{1-\frac{\lambda}{2n}}\\
=1+\frac{\lambda}{8n} \frac{x^2}{(1-\frac{\lambda}{2n})}+\frac{\frac{\lambda}{2n}}{\left(\frac{\lambda}{2n}-1\right)}x^3f_*(\xi(x)), \label{eqn:fstar_taylor}
\ee
for some $|\xi(x)|\leq |x|$ and
\be 
f_*(\xi)= \frac{2\left(e^{\xi}\left(11 e^{\xi}-11 e^{2\xi}+e^{3\xi}-1\right)\right)}{3(e^{\xi}+1)^5}.
\ee
%
%Therefore
%\be 
%f^*(A)=1+\frac{\lambda}{8n}\times \frac{A^2}{1-\frac{\lambda}{2n}}+O(A^3). \label{eqn:fstar_expansion}
%\ee
Next, we analyze the success probability for the distribution of $X$ in \eqref{eqn:secondmoment_simplified}.
\be 
\delta_n&:=\frac{f(A)\frac{\lambda}{2n}}{g(A)\psi(A,\lambda, n)}
%&=\frac{\frac{2\lambda}{n}\frac{e^{2A}}{(1+e^A)^2}\left(1-\frac{\lambda}{2n}\right)^2}{\frac{2\lambda}{n}\frac{e^{2A}}{(1+e^A)^2}\left(1-\frac{\lambda}{2n}\right)^2+\left(1-\frac{\lambda}{2n}\right)\left(1-\frac{\lambda}{n}\frac{e^A}{1+e^A}\right)}\\
=\frac{\frac{2\lambda}{n}\frac{e^{2A}}{(1+e^A)^2}}{f^*(A)}
=\frac{\frac{2\lambda}{n}\frac{e^{2A}}{(1+e^A)^2}}{1+\frac{\lambda A^2}{8n\left(1-\frac{\lambda}{2n}\right)}+A^3f_*(\xi(A))}.%=\frac{\frac{2\lambda}{n}\left(\frac{1}{4}+\frac{A}{4}+\frac{A^2}{16}+O(A^3)\right)}{1+\frac{\lambda}{8n}\times \frac{A^2}{1-\frac{\lambda}{2n}}+O(A^3)}=\frac{\lambda}{2n}\left(1+A\right)(1+o(1)).
\ee
Moreover, using Taylor expansion for $f$, we have, 
\be 
\frac{e^{2x}}{(1+e^x)^2}=\frac{1}{4}+\frac{x}{4}+\frac{x^2}{16}+x^3g_*(\zeta(x))
\ee
for some $|\zeta(x)|\leq |x|$ and
\be 
g_*(\zeta)=\frac{e^{2\zeta}\left(e^{2\zeta}-7e^{\zeta}+4\right)}{3(e^{\zeta}+1)^5}.
\ee
Therefore, combining everything, we have, 
\be 
\delta_n&=\frac{2\lambda}{n}\,\,\cdot\,\frac{\frac{1}{4}+\frac{A}{4}+\frac{A^2}{16}+A^3g_*(\zeta(A))}{1+\frac{\lambda A^2}{8n\left(1-\frac{\lambda}{2n}\right)}+A^3f_*(\xi(A))}
%&=\frac{2\lambda}{n}\left(\frac{1}{4}+\frac{A}{4}+\frac{A^2}{16}+A^3g_*(\zeta(A))\right)\left(1-\left(\frac{\lambda A^2}{8n\left(1-\frac{\lambda}{2n}\right)}+A^3f_*(\xi(A))\right)+\left(\frac{\lambda A^2}{8n\left(1-\frac{\lambda}{2n}\right)}+A^3f_*(\xi(A))\right)h_*(A\right)\\
%&=\frac{\lambda}{2n}\left(1+A+\frac{A^2}{4}+\frac{A^3}{4}g_*(\zeta(A))\right)\left(1-\left(\frac{\lambda A^2}{8n\left(1-\frac{\lambda}{2n}\right)}+A^3f_*(\xi(A))\right)+\left(\frac{\lambda A^2}{8n\left(1-\frac{\lambda}{2n}\right)}+A^3f_*(\xi(A))\right)h_*(A\right)\\
=\frac{\lambda}{2n}\left(1+A(1+b_nA)\right),
\ee
for $\{b_n\}_{n\geq 1}$ such that $b_n\leq C$ for sufficiently large $n$ and some universal constant $C$.
%where $|h_*(A)|\leq \left|\frac{\lambda A^2}{8n\left(1-\frac{\lambda}{2n}\right)}+A^3f_*(\xi(A))\right|$
%Therefore,
%\be 
%\ &\P\left(X\leq (n-s)\frac{\lambda}{2n}+\sqrt{2\log{n}}\sqrt{(n-s)\frac{\lambda}{2n}\left(1-\frac{\lambda}{2n}\right)}\right)\\
%&=\P\left(\frac{X-(n-2s+Z)\delta_n}{\sqrt{(n-2s+Z)\delta_n\left(1-\delta_n\right)}}\leq \frac{(n-s)\frac{\lambda}{2n}-(n-2s+Z)\delta_n+\sqrt{2\log{n}}\sqrt{(n-s)\frac{\lambda}{2n}\left(1-\frac{\lambda}{2n}\right)}}{\sqrt{(n-2s+Z)\delta_n\left(1-\delta_n\right)}}\right).
%\ee
Now we make use of the following lemma
\begin{lemma}\label{lemma:binomial_tail_forlowerbound} There exists a sequence  $\kappa_n(Z)$ such that $\limsup\limits_{n \rightarrow \infty}\sup\limits_{Z\leq s}\kappa_n(Z)=0 $ and
\be
\ &\frac{(n-s)\frac{\lambda}{2n}-(n-2s+Z)\delta_n+\sqrt{2\log{n}}\sqrt{(n-s)\frac{\lambda}{2n}\left(1-\frac{\lambda}{2n}\right)}}{\sqrt{(n-2s+Z)\delta_n\left(1-\delta_n\right)}}
\\&=\sqrt{2\log{n}}\left(1-\sqrt{\frac{C^*}{4(1-\theta)}}+\kappa_n(Z)\right).
\ee
\end{lemma}

The proof will be deferred to Appendix \ref{sec:technical_lemmas}. 
Applying Lemma \ref{lemma:binomial_tail_forlowerbound} along with Lemma \ref{lemma:binomial_collection}  Part (a, ii) we have 
\be 
\ & \P\left(X\leq (n-s)\frac{\lambda}{2n}+\sqrt{2\log{n}}\sqrt{(n-s)\frac{\lambda}{2n}\left(1-\frac{\lambda}{2n}\right)}\right)\\
%&=\P\left(\frac{X-(n-2s+Z)\delta_n}{\sqrt{(n-2s+Z)\delta_n\left(1-\delta_n\right)}}\leq \frac{(n-s)\frac{\lambda}{2n}-(n-2s+Z)\delta_n+\sqrt{2\log{n}}\sqrt{(n-s)\frac{\lambda}{2n}\left(1-\frac{\lambda}{2n}\right)}}{\sqrt{(n-2s+Z)\delta_n\left(1-\delta_n\right)}}\right)\\
&=\P\left(\frac{X-(n-2s+Z)\delta_n}{\sqrt{(n-2s+Z)\delta_n\left(1-\delta_n\right)}}\leq \sqrt{2\log{n}}\left(1-\sqrt{\frac{C^*}{4(1-\theta)}}+\kappa_n\right)\right)\\
&\leq \one_{C^*\leq 4(1-\theta)}+\exp\left(-\log{n}\left(1-\sqrt{\frac{C^*}{4(1-\theta)}}+\kappa_n(Z)\right)^2\right)\one_{C^*> 4(1-\theta)}\\&:=p_n(C^*,Z).
\ee
Therefore by \eqref{eqn:changeofmeasure_inyourface} and \eqref{eqn:fstar_taylor}
\be 
\ &\Ezero\left(L_{S_1}L_{S_2}\one_{\Gamma_{S_1}\cap \Gamma_{S_2}}\right)\leq (1 + o(1)) f^*(A)^{Z(n-2s+Z)}\left(p_n(C^*,Z) \right)^Z.
\ee
Using \eqref{eqn:fstar_taylor}, we note that  $f_*$ is uniformly bounded by a universal constant, and thus for sufficiently large $n$, $f^*(A)\geq 1$. Consequently, $n$ sufficiently large, 
\be
{f^*(A)}^{Z(n-2s+Z)} \leq  {f^*(A)}^{nZ}. 
\ee
%\be 
 %\left(1+\frac{\lambda}{8n} \frac{A^2}{(1-\frac{\lambda}{2n})}+\frac{\frac{\lambda}{2n}}{\left(\frac{\lambda}{2n}-1\right)}A^3f_*(\xi(A))\right)^{Z(n-2s+Z)}\leq \left(1+\frac{\lambda}{8n} \frac{A^2}{(1-\frac{\lambda}{2n})}+\frac{\frac{\lambda}{2n}}{\left(\frac{\lambda}{2n}-1\right)}A^3f_*(\xi(A))\right)^{nZ}.
%\ee
Hereafter, we divide our analysis into two cases according to the value of $C^*/(4(1-\theta))$.
\subsubsection{Case I: $C^*\leq 4(1-\theta)$:} In this case $p_n(C^*,Z)=1$.  Therefore, using \eqref{eqn:fstar_taylor}, we have, 
\be 
\Ezero\left(L_{S_1}L_{S_2}\one_{\Gamma_{S_1}\cap \Gamma_{S_2}}\right) &\leq  (1+ o(1)) (f^*(A))^{nZ}\\
%\left(1+\frac{\lambda}{8n}\times \frac{A^2}{1-\frac{\lambda}{2n}}+\frac{\frac{\lambda}{2n}}{\left(\frac{\lambda}{2n}-1\right)}A^3f_*(\xi(A))\right)^{nZ}\\
\leq (1+ o(1)) & \exp\left(Zn\left(\frac{\lambda}{8n} \frac{A^2}{(1-\frac{\lambda}{2n})}+\frac{\frac{\lambda}{2n}}{\left(\frac{\lambda}{2n}-1\right)}A^3f_*(\xi(A))\right)\right).
\ee
The uniform universal upper bound on  $f_*$ implies that for $n$ sufficiently large, 
\be
f^*(A)- 1 = \frac{\lambda}{8n}\frac{A^2}{(1-\frac{\lambda}{2n})}+\frac{\frac{\lambda}{2n}}{\left(\frac{\lambda}{2n}-1\right)}A^3f_*(\xi(A))\geq 0.
\ee
Now, we note that $Z \sim Hypergeometric(n,s,s)\stackrel{d}{\lesssim} W\sim \Bin\left(s,\frac{s}{n-s}\right)$ where $\stackrel{d}{\lesssim}$ denotes stochastic ordering. As a result, for sufficiently large $n$
\be 
 \E_{S_1,S_2}\Ezero\left(L_{S_1}L_{S_2}\one_{\Gamma_{S_1}\cap \Gamma_{S_2}}\right) 
&\leq  (1 + o(1)) \E\exp\left(Wn (f^*(A) -1) \right)\\
%&=\left(1+\frac{s}{n-s}\left\{\exp\left(n\left(\frac{\lambda}{8n} \frac{A^2}{(1-\frac{\lambda}{2n})}+\frac{\frac{\lambda}{2n}}{\left(\frac{\lambda}{2n}-1\right)}A^3f_*(\xi(A)\right)\right)-1\right\}\right)^s\\
%& \leq \exp\left[\frac{s^2}{n-s}\left\{\exp\left(n\left(\frac{\lambda}{8n}\times \frac{A^2}{1-\frac{\lambda}{2n}}+\frac{\frac{\lambda}{2n}}{\left(\frac{\lambda}{2n}-1\right)}A^3f_*(\xi(A)\right)\right)-1\right\}\right]\\
&=(1+ o(1)) \exp\left(n^{1-2\alpha+\frac{C^*}{8(1-\theta)}+o(1)}\right).
\ee
We note that $1-2\alpha+\frac{C^*}{8(1-\theta)}<0$ if $C^*<16(1-\theta)\left(\alpha-\frac{1}{2}\right)$ and in this case, we have the desired result. This is indeed true when $\alpha\leq  \frac{3}{4}$. Therefore, we are left with the region when $\alpha>\frac{3}{4}$ and $16(1-\theta)\left(\alpha-\frac{1}{2}\right)\leq C^*< 16(1-\theta)\left(1-\sqrt{1-\alpha}\right)^2.$ In this region $C^* > 4(1-\theta)$ which corresponds to our next subsection.
 \subsubsection{Case II: $C^*>4(1-\theta)$:} As deduced in the last subsection, it remains to consider the case $C^* > 4(1-\theta)$, which immediately implies $\alpha>\frac{3}{4}$. \\
 
 In this case $p_n(C^*,Z)=\exp\left(-\log{n}\left(1-\sqrt{\frac{C^*}{4(1-\theta)}}+\kappa_n(Z)\right)^2\right)$.  Therefore
 \be 
 \ & \Ezero\left(L_{S_1}L_{S_2}\one_{\Gamma_{S_1}\cap \Gamma_{S_2}}\right)
 \\ &\leq (1+ o(1)) \left(f^*(A)\right)^{nZ}\exp\left(-Z\log{n}\left(1-\sqrt{\frac{C^*}{4(1-\theta)}}+\kappa_n(Z)\right)^2\right)\\
 &\leq (1+ o(1)) \exp(Zn(f^*(A)-1))\exp\left(-Z\log{n}\left(1-\sqrt{\frac{C^*}{4(1-\theta)}}+\kappa_n(Z)\right)^2\right).
 %&=\exp\left\{Z\log{n}\left(\begin{array}{c}\frac{C^*}{8(1-\theta)}-\left(1-\sqrt{\frac{C^*}{4(1-\theta)}}\right)^2\\
 %+\frac{\sqrt{\frac{C^*\log{n}}{\lambda}}f_*(\xi(A))}{\frac{\lambda}{2n}-1}-\kappa_n^2(Z)+2\left(1-\sqrt{\frac{C^*}{4(1-\theta)}}\right)\kappa_n(Z)+\frac{C^*}{8}\left(\frac{1}{1-\frac{\lambda}{2n}}-\frac{1}{1-\theta}\right)\end{array}\right)\right\}\\
 \ee
 We note that $\limsup\limits_{n \rightarrow \infty}\sup\limits_{Z\leq s}\kappa_n(Z)=0 $ by Lemma \ref{lemma:binomial_tail_forlowerbound}, and $\lim_{n\rightarrow \infty}\frac{\lambda}{2n}=\theta$ with $\lambda\gg \log{n}$. Thus we have, using \eqref{eqn:fstar_taylor}, for any $\epsilon>0$, there exists an $n_{\epsilon}$ such that for all $n \geq n_{\epsilon}$
% \be 
% \ & \exp\left\{Z\log{n}\left(\begin{array}{c}\frac{C^*}{8(1-\theta)}-\left(1-\sqrt{\frac{C^*}{4(1-\theta)}}\right)^2\\
% 	+\frac{\sqrt{\frac{C^*\log{n}}{\lambda}}f_*(\xi(A))}{\frac{\lambda}{2n}-1}-\kappa_n^2(Z)+2\left(1-\sqrt{\frac{C^*}{4(1-\theta)}}\right)\kappa_n(Z)+\frac{C^*}{8}\left(\frac{1}{\frac{1-\lambda}{2n}}-\frac{1}{1-\theta}\right)\end{array}\right)\right\}\\
%  & \leq \exp\left\{Z\log{n}\left(\frac{C^*}{8(1-\theta)}-\left(1-\sqrt{\frac{C^*}{4(1-\theta)}}\right)^2+\epsilon\right)\right\}.
% \ee 
% Therefore for any $\epsilon>0$ we have for all $n \geq n_{\epsilon}$
 \be 
 \E_{S_1,S_2}\Ezero\left(L_{S_1}L_{S_2}\one_{\Gamma_{S_1}\cap \Gamma_{S_2}}\right)&\leq (1+o(1)) \E_{S_1,S_2}\exp\left\{Z\log{n}\left(\psi(C^*)+\epsilon\right)\right\} 
 \ee
with $\psi(\cdot)$ defined as follows. For $x >0$, we set 
 \be 
 \psi(x)=\frac{x}{8(1-\theta)}-\left(1-\sqrt{\frac{x}{4(1-\theta)}}\right)^2. 
 \ee
 We note that, $\psi'(x) >0$ 
% \be 
 %\psi'(x)=\frac{1}{2\sqrt{1-\theta}}\left(\frac{1}{\sqrt{x}}-\frac{1}{16(1-\theta)}\right)>0 
 %\ee
whenever $0<x<16(1-\theta)$. As a result, for $4(1-\theta)<C^*<16(1-\theta)$
 \be 
 \psi(C^*)\geq \psi(4(1-\theta))=\frac{1}{2}>0.
 \ee
% Since $4(1-\theta)<C^*<16(1-\theta)$ we have $\frac{C^*}{8(1-\theta)}-\left(1-\sqrt{\frac{C^*}{4(\theta)}}\right)^2\geq \frac{1}{2}$. 
 Therefore, for any $\epsilon>0$, for all $n \geq n_{\epsilon}$, we have, 
 \be 
 \ &\E_{S_1,S_2}\Ezero\left(L_{S_1}L_{S_2}\one_{\Gamma_{S_1}\cap \Gamma_{S_2}}\right)
\\& \leq (1 + o(1))  \E\exp\left\{Z\log{n}\left(\psi(C^*)+\epsilon\right)\right\}\\
 &\leq (1+ o(1)) \E\exp\left\{W\log{n}\left( \psi(C^*)+\epsilon\right)\right\} \\ %\quad \text{where} \ W\sim \Bin\left(s,\frac{s}{n-s}\right)\\
 %&=\left(1+\frac{s}{n-s}\left(\exp\left\{\log{n}\left(\frac{C^*}{8(1-\theta)}-\left(1-\sqrt{\frac{C^*}{4(1-\theta)}}\right)^2+\epsilon\right)\right\}-1\right)\right)^s\\
 & \leq (1+o(1)) \exp\left\{n^{1-2\alpha+\frac{C^*}{8(1-\theta)}-\left(1-\sqrt{\frac{C^*}{4(1-\theta)}}\right)^2+\epsilon} \right\},
\ee
where $W\sim \Bin\left(s,\frac{s}{n-s}\right)$. The second inequality above follows from the fact that $Z \stackrel{d}{\lesssim} W$. 
Now, we note that 
\be 
\ &1-2\alpha+\frac{C^*}{8(1-\theta)}-\left(1-\sqrt{\frac{C^*}{4(1-\theta)}}\right)^2
%&=1-2\alpha-\frac{C^*}{8(1-\theta)}-1+\sqrt{\frac{C^*}{(1-\theta)}}\\
\\&=2\left\{(1-\alpha)-\left(1-\sqrt{\frac{C^*}{16(1-\theta)}}\right)^2\right\}.
\ee
Since $C^*<16(1-\theta)(1-\sqrt{1-\alpha})^2$ we have that $(1-\alpha)<\left(1-\sqrt{\frac{C^*}{16(1-\theta)}}\right)^2$. Therefore there exists $\delta>0$ such that $1-2\alpha+\frac{C^*}{8(1-\theta)}-\left(1-\sqrt{\frac{C^*}{4(1-\theta)}}\right)^2<-\delta$. Choosing $\epsilon=\delta/2$ we have for all $n \geq n_{\delta/2}$
\be 
n^{1-2\alpha+\frac{C^*}{8(1-\theta)}-\left(1-\sqrt{\frac{C^*}{4(1-\theta)}}\right)^2+\epsilon}\leq n^{-\delta/2},
\ee
which therefore completes the proof of \eqref{eqn:secondmoment_toshow}.
\end{proof}

\subsection{Proof of Theorem \ref{thm:max_degree_test}} %\textcolor{red}{To include completed proof}.
%\subsection{Proof of Theorem \ref{thm:max_degree_test}}
\begin{proof}[Proof of \ref{thm:max_degree_test}\ref{thm:max_degree_upper}]
Using monotonicity arguments without loss of generality one can consider the alternative  $\Pbeta$ where $\beta$ is given by
\begin{align*}
\beta_i=A\text{ for }i\in S,\quad
0\text{ otherwise},
\end{align*}
where $A=\sqrt{C^*\frac{\log{n}}{\lambda}}$ for some $C^*$ with  $$16(1-\theta)\geq  C^*>\cmax(\alpha):=16(1-\theta)(1-\sqrt{1-\alpha})^2,$$ and $|S|=s=n^{1-\alpha}$.
Given $C^*$ let $\delta>0$ be such that 
$$C^*>16(1-\theta)[\sqrt{1+\delta}-\sqrt{1-\alpha}]^2.$$
 and let $\phi_n$ be the sequence of tests which rejects when $\max_{i\in [n]}d_i>k_n(\delta)$, and accepts otherwise, where \be p_n&:=\frac{\lambda}{2n}.\\
  k_n(\delta) &:=np_n +\sqrt{2(1+\delta)np_n(1-p_n)\log{n}}\\&=\frac{\lambda}{2}+\sqrt{(1+\delta)\lambda\Big(1-\frac{\lambda}{2n}\Big)\log{n}}.
  \ee
Thus, using FKG inequality gives
\begin{align*}
1- \Ezero\phi_n
\ge \Pzero(d_1\le k_n(\delta))^n
\ge \Big(1-n^{-(1+\delta+o(1)}\Big)^n,
\end{align*}
where the last inequality uses  Part (a, ii) of Lemma \ref{lemma:binomial_collection}. Since the RHS above converges to $1$, it is enough to show that $$\sup_{\beta\in \Xi(s,A)}\Pbeta(\max_{i=1}^{n} d_i\leq k_n(\delta))\stackrel{n\rightarrow\infty}{\rightarrow}0.$$ 
 
 With
 $d_1',\cdots,d_{s}'\stackrel{i.i.d.}{\sim}\Bin(n-s,p_n')$ with $p_n':=\frac{\lambda}{n}\frac{e^{A}}{1+e^{A}}$,  it is easy to see that $\max_{i=1}^{n} d_i$ is stochastically larger than $\max_{i=1}^{s} d_i'$, and so it suffices to show that 
 \begin{align}
 \Pbeta(\max_{i=1}^{s} d_i'\leq k_n(\delta))=\P(d_1'\le k_n(\delta))^s=(1-\P(d_1'> k_n(\delta)))^s\stackrel{n\rightarrow\infty}{\rightarrow}0.\label{eq:max_deg_attain}
 \end{align}
  To this effect, note that
\begin{align*}
\lim_{n\rightarrow\infty}\frac{k_n(\delta)-(n-s)p_n'}{\sqrt{(n-s)p_n'(1-p_n')\log (n-s)}}=-\frac{\sqrt{2C^*}}{4\sqrt{1-\theta}}+\sqrt{2(1+\delta)}.
\end{align*} 
Also note that the assumption $C^*\le 16(1-\theta)$ implies the limit above is positive.  
 Since $s=n^{1-\alpha}$ and $(1-x)=\exp(-x+O(x^2))$ as $x \rightarrow 0$, \eqref{eq:max_deg_attain} will follow from Lemma \ref{lemma:binomial_collection} Part (a, ii) if we can show that 
 \begin{align*}
& \Big[-\frac{\sqrt{2C^*}}{4\sqrt{1-\theta}}+\sqrt{2(1+\delta)}\Big]^2< 2(1-\alpha)\\
 \Leftrightarrow &-\frac{\sqrt{C^*}}{4\sqrt{1-\theta}}+\sqrt{1+\delta}< \sqrt{1-\alpha}\\
 \Leftrightarrow &\frac{\sqrt{C^*}}{4\sqrt{1-\theta}}>\sqrt{1+\delta}-\sqrt{1-\alpha},
 \end{align*}
 which holds by choice of $\delta$. This completes the proof of the upper bound.
 \end{proof}

 \begin{proof}[Proof of \ref{thm:max_degree_lower}]
 
 To show the lower bound, again consider an alternative of the form $\P_{\beta,\lambda}$, where $\beta$ is given by 
 $$\beta_i=A\text{ for }i\in S,\quad 0\text{ otherwise,}$$
 where $A=\sqrt{C^*\frac{\log n}{\lambda}}$ for some positive $C^*$ with
 $$C^*<\cmax(\alpha):=16(1-\theta)(1-\sqrt{1-\alpha})^2.$$
 and $|S|=s=n^{1-\alpha}$.
 
 Suppose, to the contrary, that there is a sequence of consistent tests based on $\max_{i\in [n]}d_i$. Thus there exists sequence of positive reals $\{k_n\}_{n\ge 1}$ such that
 $$\lim_{n\rightarrow\infty}\Pzero(\max_{i\in [n]}d_i\le k_n)=1,\quad \lim_{n\rightarrow\infty}\Pbeta(\max_{i\in [n]}d_i\le k_n)=0.$$
 
 % Invoking FKG, we have
 % $$ \lim_{n\rightarrow\infty}\prod_{i\in [n]}\P_{\beta,\lambda}(d_i\le  k_n)\le \lim_{n\rightarrow\infty}\P_{\beta,\lambda}(\max_{i\in [n]}d_i\le k_n)=0.$$
 %Let $U_1,\cdots,U_n\stackrel{i.i.d.}{\sim}\Bin(n,p_n)$ and $V_1,\cdots,V_n\stackrel{i.i.d.}{\sim}\Bin(n,p_n')$ be mutually independent, where $p_n:=\frac{\lambda}{2n}$ and $p_n':=\frac{\lambda e^A}{n(1+e^A)}$.
 
 %Suppose $k_n=np_n+\delta_n\sqrt{\log{n}}\sqrt{np_nq_n}$, and let $\delsup=\limsup\delta_n$, $\delinf=\liminf\delta_n$. 
 
 Denote $p_n=\lambda/2n$, $q_n=1-p_n$ and consider the sequence $\delta_n$ such that
 \be
 k_n&=np_n+(2np_nq_n\log{n})^{1/2}\left(1-\frac{\log\log{n} + \log (4 \pi)} {4\log{n}}+\frac{\delta_n}{2 \log n} \right).
 \ee
{ We first claim that $\delta_n \rightarrow \infty$. The proof of the claim follows immediately since by Corollary 3.4 of \cite{bollobas}, we have for any fixed $\delta\in \mathbb{R}$ that
	\be 
	\ &\Pzero\left(\dmax<np_n+(2np_nq_n\log{n})^{1/2}\left(1-\frac{\log\log{n}+ \log(4 \pi)}{4\log{n}}+\frac{\delta}{2\log{n}}\right)\right)\\ &\rightarrow e^{-e^{-\delta}}, \label{eqn:maxdegree_null_distribution}
	\ee
	whenever $\frac{np_nq_n}{(\log{n})^3}\rightarrow \infty$, a condition that holds by our assumption of $\lambda \gg (\log{n})^3$. Indeed, if $\delta_n \leq M$ for some $ 0 < M < \infty$, then $$\limsup \Pzero(\dmax \le k_n) \leq e^{{e}^{-M}} <1.$$ }
 
% Also, we note that if $-\delta_n\gtrsim \log{n}$ then for sufficiently large $n$ (possibly along a subsequence) $$(2np_nq_n\log{n})^{1/2}\left(1-\frac{\log\log{n}}{4\log{n}}-\frac{\log(2\pi^{1/2})}{2\log{n}}+\frac{\delta_n}{2\log{n}}\right)\leq (2(1-\varepsilon)np_nq_n\log{n})$$
%  for some $\varepsilon>0$. But by Lemma \ref{lemma:binomial_collection} Part (a, ii) and union bound we have that $\Pzero(\dmax>k_n)\rightarrow 1$ in that case, and thereby leading to a contradiction. Therefore we might safely assume that \be 
%  \ &(2np_nq_n\log{n})^{1/2}\left(1-\frac{\log\log{n}}{4\log{n}}-\frac{\log(2\pi^{1/2})}{2\log{n}}+\frac{\delta_n}{2\log{n}}\right)\\ &=(2(1+o(1))np_nq_n\log{n})^{1/2}.
%  \ee

  We now show that for any such choice of $k_n$ the probability of making a type II error converges to $1$. To this end note that letting $S:=\{l:\beta_l\neq 0\}$, by union bound we have the following inequality for $i \in S$ 
 \be 
 \Pbeta\left(\dmax>k_n\right)&\leq s\Pbeta\left(d_i>k_n\right)+\Pbeta\left(\max\limits_{j \in S^c}d_j>k_n\right).
 \ee
We now show that individually, the two summands in the last display converges to $0$. First, if $i \in S:=\{l:\beta_l\geq 0
\}$, we have for $X_n' \perp Y_n'$ with $X_n' \sim \Bin\left(s-1,\frac{\lambda}{n}f(2A)\right)$, $Y_n' \sim \Bin\left(n-s,\frac{\lambda}{n}f(A)\right)$ (where as usual $f(x)=\frac{e^x}{1+e^x}$)
\be 
\ & s\Pbeta\left(d_i> k_n\right)
=s\Pbeta(X_n'+Y_n'> k_n)
%&=1-\Pbeta(X_n+Y_n> k_n)\\
%&\leq \Pbeta(X_n+Y_n> np_n+\delta\sqrt{\log{n}}\sqrt{np_nq_n}) \quad \text{choosing an appropriate subsequence}\\
 \leq n^{1-\alpha-\left(1-\sqrt{C^*/16(1-\theta)}\right)^2+o(1)}=o(1).
\ee
The last inequality follows by calculations similar to those leading to proof of Lemma \ref{lemma:hc_main} upon invoking Lemma \ref{lemma:binomial_collection} Part (b, ii) and the last equality follows by the property of $C^*<\cmax(\alpha)$.

The control of $\Pbeta\left(\max\limits_{j \in S^c}d_j>k_n\right)$ is in philosophy similar to that of understanding the null behavior of $\dmax$. However, one needs to carefully overcome the contamination by signals in each of the degrees involved. In particular, for any sequence $k_n'$,
\be 
\Pbeta\left(\max\limits_{j \in S^c}d_j>k_n\right) &= \P_{\beta} \left( \max_{j \in S^c} (X_j + Y_j) > k_n \right) \nonumber\\
&\leq \P_{\beta}\left( \max_{j \in S^c} X_j + \max_{j \in S^c} Y_j  > k_n \right) \nonumber\\
&\leq \Pbeta(\max\limits_{i=1}^n X_i>k_n')+\Pbeta(\max\limits_{i=1}^n Y_i>k_n-k_n'), \\ \label{eqn:hopefully_lasthurdle}
\ee
where $X_i \stackrel{i.i.d.}{\sim} \Bin(s,\frac{\lambda}{n}f(A))$ and $Y_i \stackrel{i.i.d.}{\sim} \Bin(n-s-1,\frac{\lambda}{n}f(0))$. 
We choose $k_n' = \frac{\sqrt{n p_n q_n } a_n}{\sqrt{2 \log n}}$ for some sequence $a_n \to \infty$ sufficiently slow, to be chosen appropriately. 
%We make the choice of $k_n'$ clear below. First since $\lambda\gg \log^3{n}$, there exists a sequence $a_n \rightarrow \infty$ such that $\lambda= a_n \log^3{n}$. Then for $b_n=\frac{a_n}{\log^2{n}}$  define $k_n'=\sqrt{nb_np_nq_n\log{n}}$. 
%
Then by union bound and Bernstein's Inequality 
\be 
\ &\Pbeta(\max\limits_{i=1}^n X_i>k_n')\\
&\leq n\Pzero\left(X_1-s\frac{\lambda}{n}f(A)>k_n'-s\frac{\lambda}{n}f(A)\right)\\
&\leq n\exp\left(-\frac{\frac{1}{2}(k_n'')^2}{s\frac{\lambda}{n}f(A)(1-\frac{\lambda}{n}f(A))+\frac{1}{3}k_n''}\right), \quad k_n''=k_n'-s\frac{\lambda}{n}f(A)\\
&\leq n\exp\left(-Ck_n'\right),
\ee
where the last inequality holds for a universal constant $C>0$ since for any $a_n \to \infty$, $k'_n\gg s\frac{\lambda}{n}f(A)$ (since $s = n^{1- \alpha}$, $\alpha > \frac{1}{2}$, and $\lambda \leq n$). 
%Since $\lambda=a_n\log^3{n}$, we have by our choice of $b_n$
%\be 
%k_n''&=\sqrt{nb_n(1+o(1))p_nq_n\log{n}}\sim\sqrt{n\frac{a_n}{(\log{n})^2}\frac{\lambda}{2n}(1-\lambda/2n)\log{n}}\\ &\gtrsim a_n \log{n}\gg \log{n}.
%\ee
Since $\lambda \gg (\log n)^3$ we have $k_n' \gg \log n$ and therefore
\be 
\Pbeta(\max\limits_{i=1}^n X_i>k_n')\rightarrow 0,\label{eqn:last_hurdle_partI}
\ee
as $n \rightarrow \infty$. Now note that by stochastic ordering
\be 
\Pbeta(\max\limits_{i=1}^n Y_i>k_n-k_n')&\leq \Pbeta(\max\limits_{i=1}^n Y_i'>k_n-k_n')
\ee
where $Y_i' \stackrel{i.i.d.}{\sim} \Bin(n,p_n)$. But,
\be 
\ & k_n-k_n'\\&=np_n+(2np_nq_n\log{n})^{1/2}\left(1-\frac{\log\log{n}}{4\log{n}}-\frac{\log(2\pi^{1/2})}{2\log{n}}+\frac{\delta_n- a_n}{2\log{n}}\right)\\
%&=np_n+(2np_nq_n\log{n})^{1/2}\left(1-\frac{\log\log{n}}{4\log{n}}-\frac{\log(2\pi^{1/2})}{2\log{n}}+\frac{\delta_n-a_n\sqrt{2}}{2\log{n}}\right).
\ee
%Since $\delta_n \rightarrow -\infty$, we have by our choice of $a_n\rightarrow \infty$ that $\delta_n-a_n\sqrt{2}\rightarrow -\infty$. 
For any sequence $\delta_n \to \infty$, we can choose $a_n \to \infty$ sufficiently slow such that $\delta_n - a_n \to \infty$ as $n\to \infty$. 
But due to convergence to a continuous distribution, the convergence in \eqref{eqn:maxdegree_null_distribution} is uniform. Therefore 
\be 
\Pbeta(\max\limits_{i=1}^n Y_i'>k_n-k_n')\rightarrow 0. \label{eqn:last_hurdle_partII}
\ee
The proof is therefore complete by combining \eqref{eqn:hopefully_lasthurdle}, \eqref{eqn:last_hurdle_partI}, and \eqref{eqn:last_hurdle_partII}.
%Similarly if $j \in S^c$, we have for $X_n \perp Y_n$ with $X_n \sim \Bin\left(s,\frac{\lambda}{n}f(A)\right)$, $Y_n \sim \Bin\left(n-s-1,\frac{\lambda}{n}f(0)\right)$ 
%\be 
%n\Pbeta\left(d_i> k_n\right)&=n\Pbeta(X_n+Y_n> k_n)\\
%%&=1-\Pbeta(X_n+Y_n> k_n)\\
%%&\geq 1-\Pbeta(X_n+Y_n> np_n+\delta\sqrt{\log{n}}\sqrt{np_nq_n}) \\
%& \leq o(1),
%\ee 
%where the last equality follows by Lemma \ref{lemma:binomial_collection}. This completes the proof of the lower bound.
 %
\end{proof}

\bibliographystyle{imsart-nameyear}
\bibliography{biblio_beta_model-2}

\appendix

\section{Proof of Binomial Lemma}\label{sec:proof_of_binomial_lemma}
\begin{proof}
	\begin{enumerate}
		\item[(a)]
		
		\item[(i)]
		The upper and lower bounds follow from \cite[Theorem 1.2]{bollobas} and \cite[Theorem 1.5]{bollobas} respectively.
		\item[(ii)]

		The upper bound follows from \cite[Theorem 1.3]{bollobas}.
		The lower bound follows from part \cite[Theorem 1.6]{bollobas} if  $n\min(p_n,1-p_n)\gg \log^3n$, and from Part (a, i) otherwise.
		
		\item[(b)]
		\item[(i)]
		%Fixing $\delta>0$  
		%%\begin{align*}
		%%\delta_n:=\frac{K\sqrt{b_np_n'\log n}+b_np_n'}{\sqrt{np_n(1-p_n)\log n}}
		%%\end{align*}
		%%note that $\lim_{n\rightarrow\infty}\delta_n=0$. 
		%and setting
		%$$t_n':=\delta\sqrt{np_n(1-p_n)\log n},\quad t_n:=np_n+C_n\sqrt{np_n(1-p_n)\log n}$$ we have
		%\begin{align*}
		%&\P(X_n+Y_n=t_n)\\
		%\le &\P(Y_n>t_n')+\max_{r=0}^{t_n'}\P(X_n=t_n-r)\sum_{r=0}^{t_n'}\P(Y_n=r)\\ 
		%\le& \P(Y_n> t_n')+\max_{r=0}^{t_n'}\P(X_n=t_n-r).
		%\end{align*}
		%For bounding the first term above, note that $\lim_{n\rightarrow\infty}(t_n'-b_np_n')=\infty$, and so $t_n'\ge 2b_np_n'$ for all $n$ large. This observation, along with an application of Bernstein's inequality gives 
		%\begin{align*}
		%\P(Y_n>t_n')\le \text{exp}\Big\{-\frac{\frac{1}{2}(t_n'-b_np_n')^2}{b_np_n'(1-p_n')+\frac{t_n'-b_np_n'}{3}}\Big\}\le e^{-\frac{3}{24}t_n'}.
		%\end{align*}
		% But then $t_n'\gg \log n$, and so we have $$\frac{1}{\log n}\log \P(Y_n>t_n')=-\infty.$$
		%Thus it suffices to control the second term. Since
		%$$t_n-t_n'=np_n+(C_n-\delta)\sqrt{np_n(1-p_n)\log n},$$ the bound then follows to note that
		%$$\max_{0\le r\le t_n'}\P(X_n=t_n-r)=\P(X_n=t_n-t_n')$$
		%where we have used the fact that $t_n-t_n'-np_n\rightarrow\infty$ and the observation that the binomial distribution is unimodal.
		%Finally, using part (e)  gives
		%$$\P(X_n=t_n-t_n')
		%=\frac{1}{\sqrt{np_n}} n^{-\frac{(C-\delta)^2}{2}},$$
		%from which the result follows since $\delta>0$ is arbitrary.
		%
		%This completes the proof.
		%

		For the upper bound, fixing $\delta>0$  
		%\begin{align*}
		%\delta_n:=\frac{K\sqrt{b_np_n'\log n}+b_np_n'}{\sqrt{np_n(1-p_n)\log n}}
		%\end{align*}
		%note that $\lim_{n\rightarrow\infty}\delta_n=0$. 
		and setting
		$$t_n':=\delta\sqrt{np_n(1-p_n)\log n},\quad t_n:=np_n+C_n\sqrt{np_n(1-p_n)\log n}$$ we have
		\begin{align*}
			&\P(X_n+Y_n=t_n)\\
			\le &\P(Y_n>t_n')+\max_{r=0}^{t_n'}\P(X_n=t_n-r)\sum_{r=0}^{t_n'}\P(Y_n=r)\\ 
			\le& \P(Y_n> t_n')+\max_{r=0}^{t_n'}\P(X_n=t_n-r).
		\end{align*}
		For bounding the first term above, note that $\lim_{n\rightarrow\infty}(t_n'-b_np_n')=\infty$, and so $t_n'\ge 2b_np_n'$ for all $n$ large. This observation, along with an application of Bernstein's inequality gives 
		\begin{align*}
			\P(Y_n>t_n')\le \text{exp}\Big\{-\frac{\frac{1}{2}(t_n'-b_np_n')^2}{b_np_n'(1-p_n')+\frac{t_n'-b_np_n'}{3}}\Big\}\le e^{-\frac{3}{24}t_n'}.
		\end{align*}
		But then $t_n'\gg \log n$, and so we have 
		\begin{align}\label{eq:negligible}
			\frac{1}{\log n}\log \P(Y_n>t_n')=-\infty.
		\end{align}
		Thus it suffices to control the second term. Since
		$$t_n-t_n'=np_n+(C_n-\delta)\sqrt{np_n(1-p_n)\log n},$$ the bound then follows to note that
		$$\max_{0\le r\le t_n'}\P(X_n=t_n-r)=\P(X_n=t_n-t_n')$$
		where we have used the fact that $t_n-t_n'-np_n\rightarrow\infty$ and the observation that the binomial distribution is unimodal.
		Finally, using part (e)  gives
		$$\P(X_n=t_n-t_n')
		=\frac{1}{\sqrt{np_n}} n^{-\frac{(C-\delta)^2}{2}+o(1)},$$
		from which the result follows since $\delta>0$ is arbitrary.
		\\

		Turning towards the lower bound, note that
		\begin{align*}
			\P(X_n+Y_n=t_n)&\ge \sum_{r=0}^{t_n'}\P(X_n=t_n-r)\P(Y_n=r)
			\\&\ge \P(X_n=t_n)\P(Y_n\le t_n'),
		\end{align*}
		as $\min_{r=0}^{t_n'}\P(X_n=t_n-r)=\P(X_n=t_n)$. Since $\P(Y_n\le t_n')$ converges to 1 (from \eqref{eq:negligible}), the result then follows on using part (e).

		\item[(ii)]

		For the upper bound, we have
		\begin{align*}
			\P(X_n+Y_n\ge t_n)\le &\P(X_n\ge t_n-t_n')+\P(Y_n\ge t_n'),
		\end{align*}
		from which one can ignore $\P(Y_n\ge t_n')$ using \eqref{eq:negligible}. Using part (a, ii) gives
		$$\P(X_n\ge t_n-t_n')\le n^{-\frac{(C-\delta)^2}{2}+o(1)},$$ from which the result follows since $\delta>0$ is arbitrary.
		
		Also in this case the lower bound follows trivially, on noting that
		$$\P(X_n+Y_n\ge t_n)\ge \P(X_n\ge t_n)$$ and using part (a).
	\end{enumerate}
\end{proof}

\section{Technical Lemmas for Lower Bound}\label{sec:technical_lemmas}

\subsection{Proof of Lemma \ref{lemma:upsilon_value}}
\label{appendix:section_upsilon}
We recall from the proof of Theorem \ref{thm:sparse}\ref{thm:sparse_hopeless}, the definitions of $f(x)=\frac{e^x}{1+e^x}$ and $$h(x)=4\mu f(c_1 x) f(c_2 x) + \frac{(1- 2 \mu f(c_1x))(1-2 \mu f(c_2 x))}{1-\mu},$$ where $0 \leq \mu\leq 1$ and $c_1,c_2>0$. Then note that $\Upsilon=\prod\limits_{j=1}^5\Upsilon_i$, where 
\be 
\Upsilon_1&=\left(4f(2A)^2\right)^{\Sigma(S_1\cap S_2)}\left(\frac{1-\frac{\lambda}{n}f(2A)}{1-\frac{\lambda}{2n}}\right)^{2({Z \choose 2}-\Sigma(S_1\cap S_2))},
\ee
\be 
\Upsilon_2&=\left(4f(A)f(2A)\right)^{\Sigma(S_1\cap S_2, S_1\Delta S_2)} \\ &\times\left(\frac{1-\frac{\lambda}{n}f(2A)}{1-\frac{\lambda}{2n}}\cdot\frac{1-\frac{\lambda}{n}f(A)}{1-\frac{\lambda}{2n}}\right)^{2Z(s-Z)-\Sigma(S_1\cap S_2, S_1\Delta S_2)},
\ee
\be
\Upsilon_3&=\left(4f(A)^2\right)^{\Sigma(S_1\cap S_2^c,S_1^c\cap S_2)}\left(\frac{1-\frac{\lambda}{n}f(A)}{1-\frac{\lambda}{2n}}\right)^{2((s-Z)^2-\Sigma(S_1\cap S_2^c,S_1^c\cap S_2))},
\ee
\be 
\Upsilon_4&=\left(4f(2A)f(0)\right)^{\Sigma(S_1\cap S_2^c)+\Sigma(S_1^c\cap S_2)}\\ &\times\left(\frac{1-\frac{\lambda}{n}f(2A)}{1-\frac{\lambda}{2n}}\right)^{2{s-Z \choose 2}-(\Sigma(S_1\cap S_2^c)+\Sigma(S_1^c\cap S_2))},
\ee
\be 
\Upsilon_5&=\left(4f(A)f(0)\right)^{\Sigma(S_1\Delta S_2,(S_1\cup S_2)^c)}\\ &\times\left(\frac{1-\frac{\lambda}{n}f(A)}{1-\frac{\lambda}{2n}}\right)^{2(s-Z)(n-2s+Z)-\Sigma(S_1\Delta S_2,(S_1\cup S_2)^c)}.
\ee
As in the proof of Theorem \ref{thm:sparse}\ref{thm:sparse_hopeless}, we have that for each realization of $S_1$ and $S_2$, $\Ezero(\Upsilon_i)$ is exactly of the form $h$, with $\mu= \lambda/2n$ and appropriate $c_1,c_2$. {We note that the $\Upsilon_i$'s are independent of each other and $\E_{\bbeta=0,\lambda}[\Upsilon_5]=1$. It follows, by arguments exactly similar to those leading to \eqref{eq:upperbound_intermediate}, that there exists universal constants $C_1, C_2>0$ such that}
\be 
\Ezero(\Upsilon)&\leq \Big(1+C_1\frac{\lambda A^2}{n}\Big)^{C_2 s^2}\leq \exp\Big(C_1C_2\frac{s^2\lambda A^2}{n}\Big)\\
&=\exp(C_1C_2C^*n^{1-2\alpha}\log{n})=1+o(1), \quad \text{since} \quad \alpha>\frac{1}{2}.
\ee
\subsection{Proof of Lemma \ref{lemma:binomial_tail_forlowerbound}}
The proof of the lemma follows from simple algebra along with the facts that $\lambda \leq n$, $s \ll \sqrt{n}$, and $\frac{1-\lambda/2n}{1-\theta}=1+o(1).$ 

\section{Proof of Proposition 6.5}\label{sec:proof_lemma_hc}

%\subsection{Proof of Lemma \ref{lemma:hc_main}}
\subsection{Proof of \eqref{eq:null_var}}
We set $$a(t)=\Pzero\left(D_1>t\right)$$ and $$b(t)=\Pzero\left(D_1>t, D_2>t\right)$$ to get  
\be 
\Vzero(HC(t))&=na(t) (1- a(t))+n(n-1)(b(t)-a^2(t)).
\ee

$D_1$ and $D_2$ have some dependence through the common edge $Y_{12}$. We decompose the probabilities according to the value attained by $Y_{12}$ and use the independence of the edges to get 
\be 
b(t)&=\Pzero\left(\frac{d_1-\frac{\lambda}{2n}(n-1)}{\sqrt{(n-1)\frac{\lambda}{2n}\left(1-\frac{\lambda}{2n}\right)}}>t, \frac{d_2-\frac{\lambda}{2n}(n-1)}{\sqrt{(n-1)\frac{\lambda}{2n}\left(1-\frac{\lambda}{2n}\right)}}>t\right)\\
&=\frac{\lambda}{2n}(a'(t))^2+\left(1-\frac{\lambda}{2n}\right)(a''(t))^2, \label{eq:decomp1} \\
%\ee
%where, 
%\be 
a'(t)&=\Pzero\left(\frac{\sum_{j \neq 2}Y_{1,j}-\frac{\lambda}{2n}(n-2)}{\sqrt{(n-2)\frac{\lambda}{2n}\left(1-\frac{\lambda}{2n}\right)}}>t\sqrt{\frac{n-1}{n-2}}-\sqrt{\frac{1-\frac{\lambda}{2n}}{(n-2) \frac{\lambda}{2n}}}\right), \label{eq:defn_a1} \\ 
a''(t)&=\Pzero\left(\frac{\sum_{j \neq 2}Y_{1,j}-\frac{\lambda}{2n}(n-2)}{\sqrt{(n-2)\frac{\lambda}{2n}\left(1-\frac{\lambda}{2n}\right)}}>t\sqrt{\frac{n-1}{n-2}}+\sqrt{\frac{ \frac{\lambda}{2n}}{(n-2)(1- \frac{\lambda}{2n})}}\right).\\ \label{eq:defn_a2}
\ee

%\be 
%a_{02}'(t)&=\Pzero\left(\frac{\sum_{j \neq 1}Y_{2,j}-\frac{\lambda}{2n}(n-2)}{\sqrt{(n-2)\frac{\lambda}{2n}\left(1-\frac{\lambda}{2n}\right)}}>t\sqrt{\frac{n-1}{n-2}}-\sqrt{\frac{1-\lambda/2n}{(n-2)\lambda/2n}}\right),
%\ee

%\be 
%a_{01}''(t)&=\Pzero\left(\frac{\sum_{j \neq 2}Y_{1,j}-\frac{\lambda}{2n}(n-2)}{\sqrt{(n-2)\frac{\lambda}{2n}\left(1-\frac{\lambda}{2n}\right)}}>t\sqrt{\frac{n-1}{n-2}}+\sqrt{\frac{\lambda/2n}{(n-2)(1-\lambda/2n)}}\right),
%\ee

%\be 
%a_{02}''(t)&=\Pzero\left(\frac{\sum_{j \neq 1}Y_{2,j}-\frac{\lambda}{2n}(n-2)}{\sqrt{(n-2)\frac{\lambda}{2n}\left(1-\frac{\lambda}{2n}\right)}}>t\sqrt{\frac{n-1}{n-2}}+\sqrt{\frac{\lambda/2n}{(n-2)(1-\lambda/2n)}}\right).
%\ee
Similarly, we condition on the value of $Y_{12}$ and use the independence of edges to get
\be 
a(t)=\frac{\lambda}{2n}a'(t)+\left(1-\frac{\lambda}{2n}\right)a''(t). \label{eq:decomp2}
\ee
%and 
%\be 
%a_{02}(t)=\frac{\lambda}{2n}a_{02}'(t)+\left(1-\frac{\lambda}{2n}\right)a_{02}''(t).
%\ee
Therefore, we have, using \eqref{eq:decomp1} and \eqref{eq:decomp2}, 
\be 
n(n-1)(b(t)-a^2(t))&=n(n-1)\frac{\lambda}{2n}\left(1-\frac{\lambda}{2n}\right)(a'(t)-a''(t))^2. \label{eq:null_cov_bound}
\ee
Now, using \eqref{eq:defn_a1} and \eqref{eq:defn_a2}, we have, 
\be  
a'(t)-a''(t)
%&=\Pzero\left(\frac{\sum_{j \neq 2}Y_{1,j}-\frac{\lambda}{2n}(n-2)}{\sqrt{(n-2)\frac{\lambda}{2n}\left(1-\frac{\lambda}{2n}\right)}}>t\sqrt{\frac{n-1}{n-2}}-\sqrt{\frac{1-\lambda/2n}{(n-2)\lambda/2n}}\right)\\ &-\Pzero\left(\frac{\sum_{j \neq 2}Y_{1,j}-\frac{\lambda}{2n}(n-2)}{\sqrt{(n-2)\frac{\lambda}{2n}\left(1-\frac{\lambda}{2n}\right)}}>t\sqrt{\frac{n-1}{n-2}}+\sqrt{\frac{\lambda/2n}{(n-2)(1-\lambda/2n)}}\right)\\ 
%&=\Pzero\left(t\sqrt{(n-1)\frac{\lambda}{2n}\left(1-\frac{\lambda}{2n}\right)}+\frac{\lambda}{2n}-1<\sum_{j \neq 2}Y_{1,j}-(n-2)\frac{\lambda}{2n}\leq t\sqrt{(n-1)\frac{\lambda}{2n}\left(1-\frac{\lambda}{2n}\right)}+\frac{\lambda}{2n}\right)\\
&=\Pzero\left(\sum_{j \neq 2}Y_{1,j}= \left \lceil(n-1)\frac{\lambda}{2n}+ t\sqrt{(n-1)\frac{\lambda}{2n}\left(1-\frac{\lambda}{2n}\right)}\,\right\rceil- 1\right)
\\&=\frac{n^{-r+o(1)}}{\sqrt{ \lambda }},
\ee
where the last line uses Part (a, i) of Lemma \ref{lemma:binomial_collection}, along with the fact that $\sum_{j\ne 2}Y_{1,j}\sim \Bin(n-2,p)$ with $p=\lambda/2n$. %(\textcolor{red}{this does not follow directly!---check!} )
%\red{Similarly,
%\be  
%\ & a_{02}'(t)-a_{02}''(t)\\&=\Pzero\left(\sum_{j \neq 1}Y_{2,j}= \left \lfloor(n-2)\frac{\lambda}{2n}+ t\sqrt{(n-1)\frac{\lambda}{2n}\left(1-\frac{\lambda}{2n}\right)}+\frac{\lambda}{2n}\right\rfloor\right).
%\ee
%Under the null hypothesis, $\sum_{j \neq 2}Y_{1,j}$ and $\sum_{j \neq 1}Y_{2,j}$  are independent $\Bin\left(n-2,\frac{\lambda}{2n}\right)$.
%\blue{ Let $$h=\left \lfloor(n-2)\frac{\lambda}{2n}+ t\sqrt{(n-1)\frac{\lambda}{2n}\left(1-\frac{\lambda}{2n}\right)}+\frac{\lambda}{2n}\right\rfloor-(n-2)\frac{\lambda}{2n}.$$
%Since $\lambda \gg \log{n}$ and $t=\sqrt{2r\log{n}}$ with $q>0$, for sufficiently large $n$ one has $(n-2)\frac{\lambda}{2n}\geq 1$ and $h (n-1)\left(1-\frac{\lambda}{2n}\right)/3\geq 1$. }
%\be 
%a_{01}'(t)-a_{01}''(t)< \frac{1}{\sqrt{2\pi pq(n-2)}}\exp\left\{-\frac{h^2}{2pq(n-2)}+\frac{h}{q(n-2)}+\frac{h^3}{p^2 (n-2)^2}\right\}.
%\ee
%}
%Since the same bound holds for $a_{02}'(t)-a_{02}''(t)$, we have

Therefore, we have, using \eqref{eq:null_cov_bound}, 
\be 
\ & n(n-1)(b(t)-a^2(t))
% &<n(n-1)\frac{\lambda}{2n}\left(1-\frac{\lambda}{2n}\right)\left(\frac{1}{\sqrt{2\pi pq(n-2)}}\exp\left\{-\frac{h^2}{2pq(n-2)}+\frac{h}{q(n-2)}+\frac{h^3}{p^2 (n-2)^2}\right\}\right)^2\\
\leq   n(n-1)\frac{\lambda}{2n}\left(1-\frac{\lambda}{2n}\right) \frac{n^{-2r+o(1)}}{\lambda} =O( n^{1-2r+o(1)}) .\\ \label{eqn:null_covariance}
\ee
Also by Lemma \ref{lemma:binomial_collection} Part (a, ii), 
\be 
a(t)=\Pzero(D_1>t)= n^{-(r+o(1))}. \label{eqn:null_var_diagonal}
\ee
Therefore, combining \eqref{eqn:null_covariance} and \eqref{eqn:null_var_diagonal} we get
\be 
\Vzero(HC(t))&=na(t) (1- a(t))+n(n-1)(b(t)-a^2(t))\\
&= n^{1-r+o(1)}. \label{eqn:null_variance}
\ee

This completes the proof of \eqref{eq:null_var}.

\subsection{Proof of \eqref{eq:alt_var}}
Recall that the alternative distribution $\Pbeta$ is such that $\beta_i=A$ for $i\in S$ and $\beta_i=0$ otherwise, where $A=\sqrt{C^*\frac{\log{n}}{\lambda}}$ with $16(1-\theta)\geq C^*>\csparse(\alpha)$, $\theta=\lim \frac{\lambda}{2n}$, $|S|=s=n^{1-\alpha}$, $\alpha \in (1/2,1)$.
We begin with the following set of notation. 
\be 
a^{(s)}(t)&=\Pbeta(D_i>t), \quad i \in S, \\
a^{(n-s)}(t) &=\Pbeta(D_i>t), \quad i \in S^c, \\
b^{(s)}(t) &=\Pbeta(D_i>t,D_j>t), \quad (i,j) \in S\times S, \\
b^{(n-s)}(t) &=\Pbeta(D_i>t,D_j>t), \quad (i,j) \in S^c\times S^c, \\
b^{(s,n-s)}(t) &=\Pbeta(D_i>t,D_j>t), \quad (i,j) \in S\times S^c,  \label{eq:notation}
%b^{(n-s,s)}(t) &=\P(D_i>t,D_j>t), \quad (i,j) \in S^c\times S. 
\ee
The variance of $HC(t)$ under any $\Pbeta$ considered above can be decomposed as follows.  
\be
\ &\mathrm{Var}_{\boldsymbol{\beta},\lambda}\left(HC(t)\right) := \sum_{i=1}^{5} T_i, \label{eq:var_decomp_five}\\
T_1 &= sa^{(s)}(t)(1-a^{(s)}(t)), \\
T_2 &= (n-s)a^{(n-s)}(t)(1-a^{(n-s)}(t)), \\
T_3 &= s(s-1)(b^{(s)}(t)-(a^{(s)}(t))^2), \\
T_4 &= (n-s)(n-s-1)(b^{(n-s)}(t)-(a^{(n-s)}(t))^2), \\
T_5 &= 2 s(n-s)(b^{(s,n-s)}(t)-a^{(s)}(t)a^{(n-s)}(t)).
\ee

The basic idea of the proof is that the diagonal terms $T_1,T_2$ dominate over the covariance terms $T_3$, $T_4$ and $T_5$. The next Lemma collects the necessary details.  

\begin{lemma}
	\label{lemma:bounds_five}
	Fix $\theta = \displaystyle\lim_{n\to \infty} \frac{\lambda}{2n}$. For $t =\lfloor  \sqrt{2 r \log n}\rfloor$ with $r > \frac{C^*}{16(1-\theta)}$, we have,
	\be
	\lim_{n\to \infty} \frac{\log T_1}{\log n} &= 1-\alpha -\frac{1}{2} \left(\sqrt{2r}-\sqrt{\frac{C^*}{8(1-\theta)}}\right)^2, \quad
	\lim_{n \to \infty} \frac{\log T_2}{\log n}  = 1-r, \\
	\lim_{n \to \infty} \frac{\log T_3}{\log n} &\leq 1-2\alpha-\left(\sqrt{2r}-\sqrt{\frac{C^*}{8(1-\theta)}} \right)^2,  \quad
	\lim_{n\to \infty} \frac{\log T_4}{\log n} \leq 1-2r, \\
	\lim_{n \to \infty}& \frac{\log T_5}{\log n} \leq 1-\alpha-\frac{1}{2}\left(\sqrt{2r}-\sqrt{\frac{C^*}{8(1-\theta)}} \right)^2- r. 
	\ee
\end{lemma}

Lemma \ref{lemma:bounds_five} along with \eqref{eq:var_decomp_five} immediately implies \eqref{eq:alt_var}. We outline the 
proof of Lemma \ref{lemma:bounds_five} in the rest of the section.

\begin{proof}[Proof of Lemma \ref{lemma:bounds_five}]
	We begin by proving the bound on $T_3$ which is the most involved and captures the idea behind the asymptotic behavior of the other terms as well. Throughout this proof, we set $f(x) = e^x/(1+e^x)$. Using a conditioning argument as in \eqref{eq:decomp1}, we have, for a pair $(i,j ) \in S\times S$, 
	\be 
	b^{(s)}(t)&=\frac{\lambda}{n}f(2A) (a^{(s)'}(t) )^2+\left(1-\frac{\lambda}{n}f(2A)\right)(a^{(s)''}(t))^2,\\
	a^{(s)'}(t)&=\Pbeta\left(\frac{\sum_{l \neq j}Y_{i,l}-\frac{\lambda}{2n}(n-2)}{\sqrt{(n-2)\frac{\lambda}{2n}\left(1-\frac{\lambda}{2n}\right)}}>t\sqrt{\frac{n-1}{n-2}}-\sqrt{\frac{1- \frac{\lambda}{2n}}{(n-2)\frac{\lambda}{2n}}}\right), \label{eq:cross1}\\
	a^{(s)''}(t)&=\Pbeta\left(\frac{\sum_{l \neq j}Y_{i,l}-\frac{\lambda}{2n}(n-2)}{\sqrt{(n-2)\frac{\lambda}{2n}\left(1-\frac{\lambda}{2n}\right)}}>t\sqrt{\frac{n-1}{n-2}}+\sqrt{\frac{\frac{\lambda}{2n}}{(n-2)(1- \frac{\lambda}{2n})}}\right). \\ \label{eq:cross2}
	\ee
	Using the same argument as \eqref{eq:decomp2}, we have, 
	\be 
	a^{(s)}(t)=\frac{\lambda}{n}f(2A)a^{(s)'}(t)+\left(1-\frac{\lambda}{n}f(2A)\right)a^{(s)''}(t). 
	\ee
	Therefore, we have, using \eqref{eq:var_decomp_five},
	\be
	T_3 %= s(s-1) (b^{(s)}(t) - (a^{(s)}(t))^2) 
	= s(s-1)\frac{\lambda}{n}f(2A)\left(1-\frac{\lambda}{n}f(2A)\right)(a^{(s)'}(t)-a^{(s)''}(t))^2. \label{eq:T_3}
	\ee
	Combining \eqref{eq:cross1} and \eqref{eq:cross2}, we have, 
	\be
	a^{(s)'}(t)-a^{(s)''}(t) &=
	\Pbeta\left(\sum_{l \neq j}Y_{i,l}= \left \lfloor(n-1)\frac{\lambda}{2n}+ t\sqrt{(n-1)\frac{\lambda}{2n}\left(1-\frac{\lambda}{2n}\right)}\right\rfloor\right). 
	\ee
	We note that $\sum_{l \neq j}Y_{i,l} \stackrel{d}{=}Z_1+Z_2$ where $Z_1 \sim \Bin\left(s-2,\frac{\lambda}{n}f(2A)\right)$, $Z_2 \sim \Bin\left(n-s,\frac{\lambda}{n}f(A)\right)$, and $Z_1$, $Z_2$ are independent random variables. Setting $t_n=\left \lfloor(n-1)\frac{\lambda}{2n}+ t\sqrt{(n-1)\frac{\lambda}{2n}\left(1-\frac{\lambda}{2n}\right)}\right\rfloor$, by direct computation, we have, 
	\be
	t_n= (n-s)\frac{\lambda}{n}f(A) + C_n \sqrt{(n-s)\frac{\lambda}{n}f(A) \left(1-\frac{\lambda}{n}f(A) \right) \log (n-s) }, 
	\ee
	such that $C_n \to \sqrt{2r} - \sqrt{\frac{C^*}{8(1-\theta)}} >0 $ as $n\to \infty$. Therefore, applying  Lemma \ref{lemma:binomial_collection} Part (b, i), we have,
	\be
	a^{(s)'}(t)-a^{(s)''}(t)
	%&=\P\left(Z_1+Z_2=t_n\right)\\
	%& \leq \max\limits_{x \in S_n(\alpha,\gamma_n)}\P\left(Z_2=t_n-x\right)+\P\left(Z_1\in S_n^c(\alpha,\gamma_n) \right)\\
	& \leq  \frac{n^{-(\sqrt{r}-\sqrt{C^*/16(1-\theta)})^2+o(1)}}{\sqrt{\lambda }}.\label{eqn:as_minus_asprime}
	\ee
	Plugging this bound back into \eqref{eq:T_3} immediately implies that 
	\be
	T_3 \leq n^{1-2\alpha - \left(\sqrt{2r} - \sqrt{C^*/(8(1-\theta))} \right)^2 + o(1)}. 
	\ee
	This completes the bound on $T_3$.  
	
	We next turn to the bound on $T_4$. We use the same argument in this case to derive the following expression for $T_4$ which is exactly  comparable to \eqref{eq:T_3}. We have,
	\be
	T_4 = (n-s)(n-s+1)\frac{\lambda}{n}f(0)\left(1-\frac{\lambda}{n}f(0)\right)(a^{(n-s)'}(t)-a^{(n-s)''}(t))^2, \\  \label{eq:T_4}
	\ee
	where $a^{(n-s)'}(t),\, a^{(n-s)''}(t)$ are defined analogous to $a^{(s)'}(t)$ and $a^{(s)''}(t)$ in \eqref{eq:cross1}, \eqref{eq:cross2} respectively.  The same argument as in the bound on $T_3$ now implies that  %\textcolor{red}{(Some details...) }
	\be
	T_4 \leq n^{1- 2r + o(1)}. 
	\ee
	This gives us the desired bound on $T_4$. 
	
	Finally, we control the last covariance term $T_5$. In this case, we fix $(i,j) \in S \times S^c$. Recall the notations introduced in \eqref{eq:notation}. 
	We have, similar to \eqref{eq:decomp1},
	\be
	b^{(s,n-s)}(t) &= \frac{\lambda}{n} f(A) a^{(s,n-s)'}(t) a^{(n-s,s)'}(t) + \left(1- \frac{\lambda}{n}f(A) \right) a^{(s,n-s)''}(t) a^{(n-s,s)''}(t), \\
	a^{(s,n-s)'}(t)&=\Pbeta\left(\frac{\sum_{l \neq j}Y_{i,l}-\frac{\lambda}{2n}(n-2)}{\sqrt{(n-2)\frac{\lambda}{2n}\left(1-\frac{\lambda}{2n}\right)}}>t\sqrt{\frac{n-1}{n-2}}-\sqrt{\frac{1- \frac{\lambda}{2n}}{(n-2)\frac{\lambda}{2n}}}\right), \label{eq:cross3}\\
	a^{(s,n-s)''}(t)&=\Pbeta\left(\frac{\sum_{l \neq j}Y_{i,l}-\frac{\lambda}{2n}(n-2)}{\sqrt{(n-2)\frac{\lambda}{2n}\left(1-\frac{\lambda}{2n}\right)}}>t\sqrt{\frac{n-1}{n-2}}+\sqrt{\frac{\frac{\lambda}{2n}}{(n-2)(1- \frac{\lambda}{2n})}}\right). \\ \label{eq:cross4}
	\ee
	$a^{(n-s,s)'}(t)$ and $a^{(n-s,s)''}(t)$ are defined by switching the roles of $i,j$ in \eqref{eq:cross3} and \eqref{eq:cross4} respectively. 
	
	Similar to \eqref{eq:decomp2}, we have, 
	\be
	a^{(s)}(t) &= \frac{\lambda}{n}f(A) a^{(s,n-s)'}(t) + \left(1- \frac{\lambda}{n}f(A) \right) a^{(s,n-s)''}(t),\\
	a^{(n-s)}(t) &= \frac{\lambda}{n} f(A) a^{(n-s,s)'}(t) + \left(1- \frac{\lambda}{n}f(A) \right) a^{(n-s,s)''}(t). 
	\ee
	Therefore, we have, 
	\be
	T_5&= 2s(n-s) \frac{\lambda}{n}f(A) \left(1- \frac{\lambda}{n} f(A) \right) \big(a^{(s,n-s)'}(t) - a^{(s,n-s)''}(t) \big)\\ &\times \big( a^{(n-s,s)'}(t) - a^{(n-s,s)''}(t) \big). \\
	\label{eq:T_5}
	\ee
	We bound $\big( a^{(s,n-s)'}(t) - a^{(s,n-s)''}(t) \big)$ and $\big( a^{(n-s,s)'}(t) - a^{(n-s,s)''}(t) \big)$ exactly as \eqref{eqn:as_minus_asprime} to obtain
	\be
	\big( a^{(s,n-s)'}(t) - a^{(s,n-s)''}(t) \big) &\leq  \frac{n^{-(\sqrt{r}-\sqrt{C^*/16(1-\theta)})^2+o(1)}}{\sqrt{\lambda }},\\
	\big( a^{(n-s,s)'}(t) - a^{(n-s,s)''}(t) \big) &\leq  \frac{n^{-r + o(1)}}{\sqrt{\lambda}}. 
	\ee
	Plugging these bounds back into \eqref{eq:T_5} completes the proof.

	It remains to study  the diagonal terms $T_1,T_2$. Recall that for $i \in S$, 
	\be
	\ & a^{(s)}(t) \\&= \Pbeta\left( d_i > (n-1) \frac{\lambda}{2n} + t \sqrt{(n-1) \frac{\lambda}{2n} \left(1- \frac{\lambda}{2n} \right)}  \right), \\
	&= \Pbeta \left(d_i > (n-s) \frac{\lambda}{n} f(A)  
	+ C_n \sqrt{(n-s) \frac{\lambda}{n}f(A) \left(1- \frac{\lambda}{n} f(A) \right) \log (n-s)}  \right)  .
	\ee
	for a sequence $C_n \to \sqrt{2r} - \sqrt{\frac{C^*}{8(1-\theta)}}  >0$. We note that $d_i = Z_1 + Z_2$, where $Z_1 \sim \Bin(s, \frac{\lambda}{n} f(2A))$ and $Z_2 \sim \Bin(n-2, \frac{\lambda}{n} f(A))$ are independent random variables. Thus using Lemma \ref{lemma:binomial_collection} part (b, ii), we have, 
	\be 
	a^{(s)}(t) = n^{- \frac{1}{2} ( \sqrt{2r} - \sqrt{\frac{C^*}{8(1-\theta)}})^2 + o(1) }.
	\ee
	Plugging this bound back into the definition of $T_1$ gives us the desired result. The proof for $T_2$ is exactly similar to that of $T_4$ and is therefore omitted.  
\end{proof}

\subsection{Proof of \eqref{eq:alt_exp}}
Recall the definition of $a(t)$ from the proof of \eqref{eq:null_var}. In this case, we have, for $i \in S, j \in S^c$, 
\be
\Ebeta\left(HC(t)\right) &= s\Pbeta\left(D_i>t\right)+(n-s)\Pbeta\left(D_j>t\right)-na(t) \\
&\geq s \left(\Pbeta\left(D_i>t\right) - \Pzero\left(D_i>t\right) \right), 
\ee
using the fact that the vertex degrees are stochastically larger under the alternative. An application of Lemma \ref{lemma:binomial_collection} Part (a, ii) and (b, ii), implies that for $r > \frac{C^*}{16(1-\theta)}$
\be
\Ebeta\left(HC(t)\right) &\geq s  \left(n^{-\left(\sqrt{2r}-\sqrt{C^*/8(1-\theta)}\right)^2/2+o(1)}-n^{-r+o(1)}\right) \\ 
&\geq n^{1-\alpha-\left(\sqrt{2r}-\sqrt{C^*/8(1-\theta)}\right)^2/2+o(1)}.
\ee
This completes the proof. 

\section{Lower Bound Argument for Maximum Degree Test}\label{sec:appendix_maxdegree}
We recall that lower bound in Theorem \ref{thm:max_degree_test} contains a gap regarding $\log{n}\ll \lambda\lesssim \log^3{n}$. In this section we show that for $\log{n}\ll \lambda \lesssim \log^3{n}$, if one considers the Maximum Degree Test that rejects when $\dmax>np_n+\sqrt{\delta_n np_nq_n\log{n}}$, where $p_n=\lambda/2n$, $q_n=1-p_n$, and $\delta_n$ is some sequence of real numbers, then such tests are asymptotically powerless as soon as $C^*<\cmax(\alpha)$ defined above, if $\limsup \delta_n \neq 2$.  The case when $\limsup \delta_n = 2$ is extremely challenging, and the result of the testing problem depends on the rate of convergence of $\delta_n$ to $2$ along subsequences.

In particular, once again consider an alternative of the form $\P_{\beta,\lambda}$, where $\beta$ is given by 
$$\beta_i=A\text{ for }i\in S,\quad 0\text{ otherwise,}$$
where $A=\sqrt{C^*\frac{\log n}{\lambda}}$ for some positive $C^*$ with
$$C^*<\cmax(\alpha):=16(1-\theta)(1-\sqrt{1-\alpha})^2.$$
and $|S|=s=n^{1-\alpha}$.

Suppose to the contrary, there exists sequence of positive reals $\{k_n\}_{n\ge 1}$ such that
$$\lim_{n\rightarrow\infty}\Pzero(\max_{i\in [n]}d_i\le k_n)=1,\quad \lim_{n\rightarrow\infty}\Pbeta(\max_{i\in [n]}d_i\le k_n)=0.$$

% Invoking FKG, we have
% $$ \lim_{n\rightarrow\infty}\prod_{i\in [n]}\P_{\beta,\lambda}(d_i\le  k_n)\le \lim_{n\rightarrow\infty}\P_{\beta,\lambda}(\max_{i\in [n]}d_i\le k_n)=0.$$
%Let $U_1,\cdots,U_n\stackrel{i.i.d.}{\sim}\Bin(n,p_n)$ and $V_1,\cdots,V_n\stackrel{i.i.d.}{\sim}\Bin(n,p_n')$ be mutually independent, where $p_n:=\frac{\lambda}{2n}$ and $p_n':=\frac{\lambda e^A}{n(1+e^A)}$.

Suppose $k_n=np_n+\delta_n\sqrt{\log{n}}\sqrt{np_nq_n}$, and let $\delsup=\limsup\delta_n$, $\delinf=\liminf\delta_n$. 

Suppose $\delsup\leq 0$. Then arguing along a subsequence
\be
\Pzero(\dmax> k_n)\geq \Pzero(d_1\leq np_n)\geq \frac{1}{4}
\ee
by looking at the median of $d_1$ under $\Pzero$, if $\lambda$ is even. 

Therefore, we can safely assume that $\exists$ $\delta>0$ such that $\delsup\geq \delta$. Subsequently, we work along the subsequence along which $\delta_n$ eventually becomes at least as large as $\delta$. 
Now if $i \in S:=\{l:\beta_l\neq 0
\}$, we have along an appropriate subsequence, for $X_n \perp Y_n$ with $X_n \sim \Bin\left(s-1,\frac{\lambda}{n}f(2A)\right)$, $Y_n \sim \Bin\left(n-s,\frac{\lambda}{n}f(A)\right)$ (where as usual $f(x)=\frac{e^x}{1+e^x}$)
\be 
\ & \Pbeta\left(d_i\leq k_n\right)=\Pbeta(X_n+Y_n\leq k_n)\\
&=1-\Pbeta(X_n+Y_n> k_n)\\
&\geq 1-\Pbeta(X_n+Y_n> np_n+\delta\sqrt{\log{n}}\sqrt{np_nq_n}) \\
& \geq 1-n^{-\left(\frac{\delta}{\sqrt{2}}-\sqrt{C^*/16(1-\theta)}\right)^2+o(1)}.
\ee
The last inequality follows by calculations similar to those leading to the behavior of $T_1$ in Lemma \ref{lemma:hc_main} upon invoking Lemma \ref{lemma:binomial_collection} Part (b, ii).

Similarly if $i \in S^c$, we have for $X_n \perp Y_n$ with $X_n \sim \Bin\left(s,\frac{\lambda}{n}f(A)\right)$, $Y_n \sim \Bin\left(n-s-1,\frac{\lambda}{n}f(0)\right)$ 
\be 
\Pbeta\left(d_i\leq k_n\right)&=\Pbeta(X_n+Y_n\leq k_n)\\
&=1-\Pbeta(X_n+Y_n> k_n)\\
&\geq 1-\Pbeta(X_n+Y_n> np_n+\delta\sqrt{\log{n}}\sqrt{np_nq_n}) \\
& \geq 1-n^{-\frac{\delta^2}{2}+o(1)}.
\ee
Therefore, by FKG inequality
\be 
\ &  \Pbeta\left(\dmax \leq k_n\right)\geq \prod\limits_{i=1}^n\Pbeta\left(d_i\leq k_i\right)\\
&\geq \exp\left((n-s)\log\left(1-n^{-\frac{\delta^2}{2}+o(1)}\right)+s\log\left(1-n^{-\left(\frac{\delta}{\sqrt{2}}-\sqrt{C^*/16(1-\theta)}\right)^2+o(1)}\right)\right)\\
&\sim \exp\left(-n^{1-\frac{\delta^2}{2}+o(1)}(1+o(1))-n^{1-\alpha-\left(\frac{\delta}{\sqrt{2}}-\sqrt{C^*/16(1-\theta)}\right)^2+o(1)}(1+o(1))\right).
\ee
Now suppose that $\delta>\sqrt{2}$. Then $1-\delta^2/2<0$ and note that $1-\alpha-\left(\frac{\delta}{\sqrt{2}}-\sqrt{C^*/16(1-\theta)}\right)^2<0$ since the function $$g(x)= 1-\alpha-\left(x-\sqrt{C^*/16(1-\theta)}\right)^2$$ is decreasing for for $x\geq \sqrt{C^*/16(1-\theta)}<1$, and $g(1)<0$. 
Therefore, $\delta\leq \sqrt{2}$. Let us consider the case when $\delta <\sqrt{2}$ first. 

Then there exists $\varepsilon>0$ such that $\delta=\sqrt{2(1-\varepsilon)}$. We now show that $\exists$ $\eta>0$ such that $\Pzero(\dmax>k_n)\geq \eta$, and thereby arriving at a contradiction. To show this we use a second moment method in conjunction with Paley-Zygmund Inequality.

In particular, define $\zeta_i=\I(d_i>k_n), \quad i=1,\ldots,n.$ Then by Paley-Zygmund Inequality
\be 
\Pzero\left(\sumn\zeta_i>\frac{1}{2}\Ezero(\sumn \zeta_i)\right)\geq \frac{1}{4}\frac{(\Ezero(\sumn\zeta_i))^2}{\Ezero(\sumn\zeta_i)^2}. \label{eqn:paleyzyg}
\ee
Now with a proof similar to that of \eqref{eq:alt_exp} in Lemma \ref{lemma:hc_main}
\be 
\Ezero(\sumn\zeta_i)&=n\Pzero(d_1>k_n)\gtrsim n^{1-(1-\varepsilon)^2+o(1)}. \label{eqn:paleyzyg_firstmoment}
\ee
Also by a proof similar to that of \eqref{eqn:null_covariance}
\be 
\ & \Ezero(\sumn\zeta_i)^2\\
&=n\Pzero(d_1>k_n)+n(n-a)(\Pzero(d_1>k_n,d_2>k_n)\\&-\Pzero(d_1>k_n)\Pzero(d_2>k_n))+n(n-1)\Pzero(d_1>k_n)\Pzero(d_2>k_n)\\
&\lesssim n^{1-(1-\varepsilon)^2+o(1)}+n^{1-2(1-\varepsilon)+o(1)}+n^{2-2(1-\varepsilon)^2+o(1)}.\label{eqn:paleyzyg_secondmoment}
\ee
Therefore, combining \eqref{eqn:paleyzyg}, \eqref{eqn:paleyzyg_firstmoment}, and \eqref{eqn:paleyzyg_secondmoment}, we have the existence of an $\eta>0$ such that
\be 
\Pzero\left(\sumn\zeta_i>\frac{1}{2}\Ezero(\sumn \zeta_i)\right)\geq \frac{4\eta}{4}=\eta. \label{eqn:paleyzyg}
\ee
This in turn implies the proof in the case of $\delta>\sqrt{2}$ since
\be 
\Pzero(\dmax>k_n)&=\Pzero(\sumn\zeta_i\geq 1)\\ &\geq \Pzero(\sumn\zeta_i\geq \frac{1}{2}\Ezero(\sumn \zeta_i))\geq \eta,
\ee
where the second to last inequality holds since by \eqref{eqn:paleyzyg_firstmoment}, $\Ezero(\sumn \zeta_i)\geq 1$ for sufficiently large $n$.

\end{document}